\documentclass[11pt]{amsart}

\usepackage{graphicx}
\usepackage{pgfplots}
\usepackage{overpic}
\usepackage{float}

\pgfplotsset{compat=1.11}

\usepackage{amssymb,amsfonts,eucal}
\usepackage{amsmath}
\usepackage{eucal}
\usepackage{mathtools}
\usepackage{pdfpages}
\usepackage{amsthm}
\usepackage{pstricks}
\usepackage{pict2e}
\usepackage{enumerate}
\usepackage[bookmarks]{hyperref}

\allowdisplaybreaks

\DeclareMathOperator{\tr}{tr}
\DeclareMathOperator{\cof}{cof}
\DeclareMathOperator{\diag}{diag}

\DeclareMathOperator{\dist}{dist}
\DeclareMathOperator{\lspan}{span}
\DeclareMathOperator{\sign}{sign}
\DeclareMathOperator{\image}{Im}

\newcommand{\omt}{\Omega(t)}
\newcommand{\rr}{\mathbb{R}}
\newcommand{\mm}{\mathbb{M}}
\newcommand{\zz}{\mathbb{Z}}
\newcommand{\cc}{\mathcal C}
\newcommand{\bb}{\mathcal B}

\newcommand{\Ell}{{\mathfrak L}}

\newcommand{\ii}{{\mathcal I}}
\newcommand{\oo}{\mathcal{O}}
\newcommand{\rrr}{\mathcal{R}}
\newcommand{\wws}{{\mathcal W}_{s}(X_3)}
\newcommand{\wwu}{{\mathcal W}_{u}(X_3)}
\newcommand{\ww}{{\mathcal W(X_3)}}
\newcommand{\eps}{\varepsilon}
\newcommand{\kk}{\kappa}
\newcommand{\khalf}{\tfrac{\kk}2}
\newcommand{\half}{{\textstyle \frac12}}

\newcommand{\bx}{\bar x}

\newcommand{\bp}{\bar p}

\newcommand{\dd}{{\mathcal D}}

\newcommand{\sltwo}{{\rm SL}(2,\rr)} 
\newcommand{\curve}{C^0(\rr,\sltwo)\cap C^2(\rr,\mm^2)}

\newcommand{\sotwo}{{\rm SO}(2,\rr)}

\newcommand{\sllatwo}{\mathfrak{sl}(2,\rr)}
\newcommand{\tasltwo}{T_A\sltwo}
\newcommand{\ip}[2]{\left\langle#1,#2\right\rangle}

\newcommand{\isp}[1]{\quad{\text{#1}}\quad}

\newtheorem{theorem}{Theorem}[section]
\newtheorem{lemma}[theorem]{Lemma}%[section]
\newtheorem{corollary}[theorem]{Corollary}%[section]

\theoremstyle{remark}
\newtheorem{remark}[theorem]{Remark}
\theoremstyle{definition}
\newtheorem{definition}[theorem]{Definition}%[section]

\numberwithin{equation}{section}

\begin{document}

\title[Affine motion of 2d incompressible fluids]
{Affine motion of 2d incompressible fluids\\ surrounded by vacuum and flows in $\sltwo$}

\author{Jay Roberts}
\address{
Department of Mathematics\\
 University of California\\
Santa Barbara, CA 93106}
\email{jayroberts91@gmail.com}

\author{Steve Shkoller}\thanks{SS was supported 
 the Department of Energy Advanced Simulation and
Computing (ASC) Program. }
\address{Department of Mathematics\\
University of California\\
 Davis, CA 95616}
\email{shkoller@math.ucdavis.edu}
 
\author{Thomas C.\ Sideris}
\thanks{TCS acknowledges helpful discussions with Professor W.-Y.\ Wong of Michigan State University.}
\address{Department of Mathematics\\
 University of California\\
Santa Barbara, CA 93106}
\email{sideris@math.ucsb.edu}

\date{\today}

\begin{abstract}
The affine motion of  two-dimensional (2d) incompressible  fluids surrounded by  vacuum can be reduced to a
completely integrable and globally solvable
Hamiltonian system of ordinary differential equations for the deformation gradient in $\sltwo$.
In the case of perfect fluids, the motion is given by geodesic flow
in $\sltwo$ with the Euclidean metric, while for magnetically conducting fluids (MHD),
the motion is governed by a harmonic oscillator  in $\sltwo$.  A complete classification
of the dynamics is given including  rigid motions, rotating eddies with stable and
unstable manifolds, and solutions with vanishing pressure.  For perfect fluids, the displacement
generically becomes unbounded, as $t\to\pm\infty$.  For MHD, solutions are bounded and  
generically  quasi-periodic and recurrent.
\end{abstract}

\maketitle

\vfill

\tableofcontents

\vfill

\clearpage

\section{Introduction}

The equations of motion for the velocity $u$, the magnetic field $b$, and the
pressure $p$ of a 2d incompressible magnetically conducting fluid  are
\begin{align}
\nonumber
&D_tu=-\nabla p+b\cdot\nabla b\\
\label{pde1}
&D_tb=b\cdot\nabla u\\
\nonumber
&\nabla\cdot u=\nabla\cdot b=0.
\end{align}
Here,
$D_t=\partial_t+u\cdot\nabla $ is the material time derivative.
The system \eqref{pde1} is to be solved
in a space-time domain of the form $\{(x,t)\in\rr^2\times\ii
: x\in\omt\}$,
where $\ii\subset \rr$ is some time interval.
The equations are supplemented by the vacuum free boundary conditions
\begin{equation}
\label{pde2}
p=0\isp{and} b\cdot n=0\isp{on}\partial\omt,
\end{equation}
where dot ``\;$\cdot$\;" denotes the inner product on $\rr^2$.
The free boundary $\partial\omt$ is also assumed to move with the fluid.
When the magnetic field $b$ vanishes identically, the system reduces to the incompressible
Euler equations, which will treated as a distinct  case.

For initial data satisfying the Rayleigh-Taylor sign condition, 
 local well-posedness for the incompressible  free boundary Euler equations with bulk vorticity was established in
 \cite{Christodoulou-Lindblad}, \cite{Lindblad-2005-2},  \cite{Coutand-Shkoller-2007}, \cite{ShZe2006},  \cite{ZhZh2008}, \cite{Coutand-Shkoller-2010} and
 for the incompressible free boundary MHD problem in  \cite{GuWa2016}, \cite{SuWaZh2017}.

In this article we explore the affine motion of incompressible planar fluids surrounded by vacuum.
An affine motion is one whose deformation and velocity gradients 
depend only on time.  
The use of affine deformations is a well-established tool in continuum mechanics,
first introduced in the context of the vacuum free boundary incompressible Euler system in \cite{sideris-2014},
\cite{sideris-2017}.
Under this assumption, the fluid equations \eqref{pde1}, \eqref{pde2}
reduce to a  globally solvable  constrained Lagrangian system of ordinary differential  equations
\eqref{mainode}
for the deformation gradient in the special linear group $\sltwo$.  The natural phase space is the 6-dimensional
tangent bundle of $\sltwo$, which we  regard as being embedded in $\rr^8$ with the Euclidean metric.
This system possesses three integrals of  motion corresponding to 
conservation of energy and  invariance under the left and right
action of the special orthogonal group $\sotwo$.   Using a suitable of  parameterization of $\sltwo$ followed by  a Legendre
transformation, we obtain an equivalent  completely integrable Hamiltonian system \eqref{Hamsys}.
 We shall then provide a complete description of all such motions in terms of the values of the invariants
in both cases: MHD and Euler.

 Taking the unit disk as the reference  domain,  the time-dependent
fluid domains arising from incompressible affine motion are ellipses of constant area.  The principle axes
of these fluid ellipses are determined by the eigendirections and eigenvalues of the stretch tensor.
Incompressible affine motion  allows for  compression along one axis and expansion along the other,
combined with rigid rotation of the reference and fluid domans.

Once the deformation gradient and fluid domains are known, the pressure, velocity, and magnetic field, if present,  are 
recovered through explicit formulae, taking into account the boundary conditions.  The sign of the pressure
is preserved by the motion.  There exist special solutions whose pressure vanishes identically.  In this case,
the equations of motion are linear, and solutions may be found explicitly.

The affine motion of incompressible perfect fluids in the plane is described by geodesic flow for the deformation
gradient in $\sltwo$, echoing the
classic result of Arnold on geodesic flow in the space of volume preserving diffeomorphims \cite{Arnold-1966}.
By energy conservation, the material velocity is bounded.
However,  the fluid domain becomes  unbounded, generically, as $|t|\to\infty$,  and its
 diameter  grows linearly in time 
approaching inifinite eccentricity.
Additionally, there is an invariant manifold of initial data 
leading to rigidly rotating fluid disks (eddies) of arbitrary angular velocity.  These solutions
are represented by  curves in $\sotwo$.
The manifold of rotational solutions is hyperbolic, possessing both stable and unstable
invariant manifolds.  Solutions on these manifolds are semi-bounded, and they decay exponentially to a rotating disk,
as $t\to\infty$ in the stable case and as $t\to-\infty$ in the unstable case.  Unbounded solutions are asymptotic to straight lines
in the space of $2\times2$ matrices.  In the special case of vanishing pressure,  solutions coincide with geodesic lines
in $\sltwo$.

The affine motion of incompressible magnetically conducting planar fluids can be viewed as a simple harmonic
oscillator, constrained to $\sltwo$, through the addition of a one-parameter restoring force to the equations of
geodesic motion.  Here, all solutions are bounded and, generically,  quasi-periodic and recurrent.
There exist rigid solutions with fluid ellipses of arbitrary eccentricity.  Included among these are rotating
disks of arbitrary angular velocity.  For sufficiently large angular velocities, the rotating disk solutions
possess a homoclinic invariant manifold.  Solutions on this manifold are asymptotic  to a pair of rotational solutions,
with an exponential convergence rate, as $t\to\pm\infty$.

The main results concerning the asymptotic behavior of solutions
for MHD and perfect fluids are given in Sections \ref{mhdsec} and \ref{pfsec}, respectively.
Up to that point, the exposition for the two cases is presented in parallel.
The initial  sections are devoted to the algebraic and geometric properties of the phase space
for the system of ordinary equations describing the motion of the deformation gradient.
In Section \ref{eam}, the primary equations \eqref{mainode} are derived from the fundamental fluid equations
\eqref{pde1}.
The global existence
theorem is presented in the next section along with a discussion of the integral invariants of this system.
We then devote considerable time developing the Hamiltonian structure
of the equations of motion.  In Section \ref{HamRed}, we  derive an equivalent completely integrable
Hamiltonian system \eqref{Hamsys}.  The qualitative behavior of this system is decisively influenced by the
values of the invariants, and in particular, the presence or absence of solutions which take values in  $\sotwo$.
These two cases are examined in Sections \ref{Hamkp} and \ref{Hamkz}, respectively.   
The  results given in Theorems \ref{mainthm1} and \ref{mainthm2}, hinge on the fact that 
the systems \eqref{Hamsys2} and \eqref{Hamsys3} each partially uncouple, %have a two-dimensional uncoupled subsystem
allowing for a complete phase plane analysis.
The remainder of the paper is devoted to cataloging the rich behavior of solutions for all admissible
initial conditions.  For example, we find solutions with vanishing pressure, rigid solutions, and 
invariant manifolds for rigid rotations, and in Section \ref{Tansp} we classify the initial
velocities in the tangent space of the initial position
which produce these distinguished solutions.   
For the convenience of the reader, a glossary of notation appears
at the end of the paper in Appendix \ref{glossary}.

\section{Matrix inner product space and  groups}

\begin{definition}
\label{2b2mat}
By $\mm^2$, we denote the set of $2\times 2$ matrices over $\rr$ with the Euclidean inner product
\[
\ip{ A}{B}=\sum_{i,j}A_{ij}B_{ij}=\tr A^\top B
\]
and norm
\[
|A|=\ip{A}{A}^{1/2}.
\]
\end{definition}

\begin{lemma}
\label{symip}
For all $A, B, C\in\mm^2$,
\[
\ip{AB}{C}=\ip{B}{A^\top C}=\ip{A}{CB^\top}.
\]

\end{lemma}

\begin{definition}
\label{zdef}
Define the following vectors in $\mm^2$:
\[
I=
\begin{bmatrix}
1&0\\0&1
\end{bmatrix},\quad
Z=
\begin{bmatrix}
0&-1\\1&0
\end{bmatrix},\quad
K=
\begin{bmatrix}
1&0\\0&-1
\end{bmatrix},\quad
M=
\begin{bmatrix}
0&1\\1&0
\end{bmatrix}.
\]

\end{definition}

\begin{lemma}
\label{symasym}
The vectors $I, Z, K, M$ are an orthogonal basis for $\mm^2$ satisfying the relations
\[
Z^2=-I,\quad K^2=M^2=I,\quad ZK= M,\quad MK= Z, \quad MZ=K.
\]

\end{lemma}

\begin{definition}
\label{sltsot}
The special linear group is given by
\[
\sltwo=\{A\in\mm^2:\det A=1\}
\]
and the special orthogonal group is
\[
\sotwo=\{U\in \sltwo:U^{-1}=U^\top\}.
\]
\end{definition}

\begin{lemma}
\label{iso}
For all $A\in\mm^2$ and $U, V\in\sotwo$,
\[
|UAV|=|A|.
\]
The left and right actions of $\sotwo$ on $\mm^2$ and on $\sltwo$ are isometries.
\end{lemma}

The subgroup $\sotwo$ will play a special role in the sequel.
Here is the first of several characterizations that we shall repeatedly use.
\begin{lemma}
\label{so2.1}
Elements of $\sotwo$ are norm minimizers in $\sltwo$.  There holds
\[
\min\{|A|^2:A\in\sltwo\}=2
\]
and
\[
\sotwo=\{A\in\sltwo: |A|^2=2\}.
\]
\end{lemma}

\begin{proof}

Given $A\in\sltwo$,  we have that
\[
\det A^\top A=1,
\]
and so the eigenvalues of the positive definite symmetric matrix $ A^\top A$ satisfy
\[
\lambda_1\ge\lambda_2>0\isp{and} \lambda_1\lambda_2=1.
\]
Therefore,
\[
|A|^2=\tr A^\top A=\lambda_1+\lambda_2\ge2,
\]
with equality if and only if $A^\top A=I$.
Using the polar decomposition, there exists $U\in\sotwo$ such that
\[
A=U(A^\top A)^{1/2}.
\]
So by Lemma \ref {iso}, $|A|^2=2$ if and only if $A=U$.
\end{proof}

\begin{definition}
\label{rotmat}
Define the one-parameter family of rotations
\[
U(\theta)=\exp(\theta Z)=\cos \theta I+\sin\theta Z=
\begin{bmatrix}
\cos\theta&-\sin\theta\\ \sin\theta&\cos\theta
\end{bmatrix},\quad
\theta\in\rr.
\]
\end{definition}

\begin{lemma}
With this definition, we have
\label{so2.2}
$
\sotwo=\left\{U(\theta)
: \theta\in\rr\right\}.
$
\end{lemma}

\begin{lemma}
\label{cofiso}
The cofactor map $\cof:\mm^2\to\mm^2$ satisfies 
\[
\cof A= ZAZ^\top=Z^\top AZ.
\]
It  is linear, symmetric, and orthogonal.
\end{lemma}

\begin{proof}
The identity is easily verified.
 By Lemma \ref{symip}, we  have that
\[
\ip{\cof A}{B}=\ip{ZAZ^\top}{B}=\ip{A}{Z^\top B Z}=\ip{A}{\cof B}.
\]
and
\[
\ip{\cof A}{\cof B}=\ip{A}{\cof\cof B}=\ip{A}{B}.
\]
\end{proof}

\begin{remark}
The linearity of the cofactor map on $\mm^2$ is a dstinguishing feature of the planar case,
and it plays a critical  role in our analysis.
\end{remark}

\begin{lemma}
\label{detcof}
For any $A\in\mm^2$, we have
\[
\half\det A = \ip {A}{\cof A}.
\]
\end{lemma}

\begin{lemma}
\label{detder}
The determinant map $\det:\mm\to\rr$ is $C^\infty$ and 
\[
\frac\partial{\partial A}\det A=\cof A.
\]
\end{lemma}

\begin{lemma}
\label{metaid}
For all $A, B\in \mm^2$, there holds
\[
\ip{A}{B}\ip{\cof A}{B}+\ip{ZA}{B}\ip{AZ}{B}
=\half\ip{\cof A}{A}|B|^2+\half\ip{\cof B}{B}|A|^2.
\]
\end{lemma}

\begin{proof}

If we expand the vectors $A, B$ in the basis of Definition \ref{zdef}
\[
A=\tfrac1{\sqrt2}(a_1I+a_2Z+a_3K+a_4M)
\isp{and}
B=\tfrac1{\sqrt2}(b_1I+b_2Z+b_3K+b_4M),
\]
then both sides are equal to
$\; 
(a_1^2+a_2^2)(b_1^2+b_2^2)-(a_3^3+a_4^2)(b_3^3+b_4^2).
$
\end{proof}

\section{The geometry of $\sltwo$}

\begin{lemma}
\label{normalvector}
The group $\sltwo$ is a smooth three-dimensional embedded submanifold of $\mm^2$.
The vector $\cof A$ is normal to $\sltwo$ at the point $A\in\sltwo$, and
the tangent space to $\sltwo$ at $A$ is
\[
T_{A}{\sltwo}=\{B\in\mm^2: \ip {B}{\cof A}=\tr BA^{-1}=0\}.
\]
For each $A\in\sltwo$, $T_A\sltwo$ is a three-dimensional  subspace  in $\mm^2$.
\end{lemma}

\begin{proof}
This follows from Definition \ref{sltsot} and Lemma \ref{detder}.
\end{proof}

\begin{definition}
\label{tanbund}
Define the tangent bundle
\[
\dd=\{(A,B)\in \mm^2\times\mm^2: A\in\sltwo,\;B\in \tasltwo\}.
\]
We also define
\[
\dd_0=\{(A,B)\in\dd:A\in\sotwo\}.
\]
\end{definition}

\begin{lemma}
$\dd$ is a smooth 6-dimensional embedded submanifold of $\mm^2\times\mm^2$.
\end{lemma}

\begin{definition}
\label{varphidef}
Define  the smooth map $\varphi:\rr^3\to\mm^2$  by
\[
\varphi(x)
=\tfrac1{\sqrt2}(\rho(x)\cos x_3 I+\rho(x)\sin x_3 Z+x_1 K + x_2 M),
\]
with
\[
\rho(x)=(2+x_1^2+x_2^2)^{1/2}.
\]
\end{definition}

\begin{remark}

We shall occasionally find it convenient to identify $\rr^3$ with $\rr^2\times\rr$ by writing
$x=(x_1,x_2,x_3)=(\bx,x_3)$.  Thus, for example, we have $\rho(x)=(1+|\bx|^2)^{1/2}$.
\end{remark}

\begin{lemma}
\label{varphiimmersion}
The map $\varphi$ is an immersion from $\rr^3$ onto $\sltwo$, 
\[
|\varphi(x)|^2=\rho(x)^2+|\bx|^2=2(1+|\bx|^2),
\]
and $\varphi(x)\in\sotwo$ if and only if 
$\bx=(x_1,x_2)=0$.
\end{lemma}

\begin{proof}
Let $A\in \mm^2$ and set
\[
\bar A=A-\tfrac1{\sqrt2}(x_1K+x_2M),\isp{with} x_1=\tfrac1{\sqrt2}\ip{A}{K}\isp{and}x_2=\tfrac1{\sqrt2}\ip{A}{M}.
\]
Then $\ip{\bar A}{K}=\ip{\bar A}{M}=0$, and so $\bar A\in\lspan\{I,Z\}$.  Thus, we can find
$\rho\ge0$ and $x_3\in\rr$ such that
\[
\bar A=\tfrac1{\sqrt2}(\rho\cos x_3 I+\rho\sin x_3 Z).
\]
We now have
\[
A=\tfrac1{\sqrt2}(\rho\cos x_3 I+\rho\sin x_3 Z+x_1K+x_2M).
\]
It follows from Lemma \ref{detcof} that
\[
\det A =\half\ip{A}{\cof A}
= \half(\rho^2-|\bx|^2).
\]
Therefore, $A\in\sltwo$ if and only if  $\rho^2=2+|\bx|^2$.
This proves that $\varphi:\rr^3\to\sltwo$ is surjective.

Since $|\varphi(x)|^2=\rho(x)^2+|\bx|^2=2(1+|\bx|^2)$, Lemma \ref{so2.1} implies that $\varphi(x)\in\sotwo$
if and only if 
$\bx=0$.

For any $x\in\rr^3$, it is straightforward to verify that the linear map $D_x\varphi(x):\rr^3\to\mm^2$ satisfies
\begin{equation}
\label{kernel}
\ker D_x\varphi(x) =\{0\}\isp{and} \image D_x\varphi(x)=(\lspan \cof \varphi(x))^\perp.
\end{equation}
By Lemma \ref{normalvector}, this says that  $D_x\varphi(x)$ is a bijection from $\rr^3$ onto $T_{\varphi(x)}\sltwo$,
thereby proving that $\varphi$ is an immersion.

\end{proof}

\begin{lemma}
\label{metric}
The map $\varphi$ defines a local coordinate chart for $\sltwo$ in a neighborhood of any point $x\in\rr^3$.
The metric on $T_{\varphi(x)}\sltwo$ in these coordinates is given by
\[
g(x)=
\begin{bmatrix}
1+x_1^2/\rho(x)^2 & x_1x_2/\rho(x)^2 & 0\\
x_1x_2/\rho(x)^2 & 1 + x_2^2/\rho(x)^2 & 0\\
0&0
&\rho(x)^2
\end{bmatrix}.
\]

\end{lemma}

\begin{proof}
That $\varphi$ defines  local coordinate charts follows from Lemma \ref{varphiimmersion}.
The form of the metric is $g(x)=D_x\varphi(x)^\top D_x\varphi(x)$, the components of which are
 $g_{ij}(x)=\ip{D_{x_i}\varphi(x)}{ D_{x_j}\varphi(x)}$.
\end{proof}

\begin{remark}
For each $r^2>1$, the sphere
$\{A\in\sltwo:|A|^2=r^2\}$ corresponds to the torus
$\{\varphi\in\rr^4:\varphi_1^2+\varphi_2^2=\half r^2+1, \; \varphi_3^2+\varphi_4^2=\half r^2-1\}$.

\end{remark}
\begin{definition}
\label{bigphidef}
Define the map $\Phi:\rr^3\times\rr^3\to\mm^2\times \mm^2$ by
\[
\Phi(x,y)=(\varphi(x),D_x\varphi(x)y).
\]
\end{definition}

\begin{lemma}
\label{bigphionto}
$\Phi$ maps $\rr^3\times\rr^3$ onto $\dd$.
\end{lemma}

\begin{proof}
This follows from Lemma \ref{varphiimmersion}.
\end{proof}

\begin{remark}
In this context, we can think of $\rr^3\times\rr^3$ as the tangent bundle of $\rr^3$,
 the linear map $D_x\varphi(x)$ as the push-forward from
$T_x\rr^3$ to $T_{\varphi(x)}\sltwo\subset\mm^2$, and the linear map
$D_x\varphi(x)^\top$ as the pull-back from $T_{\varphi(x)}\sltwo$ to $T_x\rr^3$.
\end{remark}

\section{The equations of affine motion}
\label{eam}
\begin{definition}
An incompressible affine motion defined on the unit disk  $\bb\subset\rr^2$
is a one-parameter family of volume preserving diffeomorphisms of the form
\[
x(t,y)=A(t)y,\quad y\in\bb,\; t\in\ii,
\]
on some interval $\ii\subset\rr$, with
\[
 A\in C^0(\ii,\sltwo)\cap C^2(\ii,\mm^2).
\]
\end{definition}

Here, $\bb$ is the reference domain, and the domain occupied by the material (fluid) at time $t$ is
the image 
\[
\omt=A(t)\bb=\{x\in\rr^2:|A(t)^{-1}x|^2\le1\}.
\]
For affine motion, $\omt$
is an ellipse centered at the origin with principal axes determined by 
the eigendirections and eigenvalues of the positive definite symmetric (stretch) matrix
$(A(t)A(t)^{\top})^{1/2}$, for all $t\in\rr$.

The spatial velocity field associated to an affine motion is
defined by 
\[
u(t,x(t,y))=\partial_tx(t,y)=\dot A(t)y,\quad y\in\bb,
\]
or equivalently
\[
u(t,x)=\dot A(t)A(t)^{-1}x,\quad x\in\omt.
\]

\begin{definition}
\label{sllatwodef}
Define the special linear Lie algebra
\[
\sllatwo=T_I\sltwo=\{L\in\mm^2:\tr L=0\}=\lspan\{K, M, Z\}.
\]
\end{definition}

\begin{lemma}
\label{phasespace}
If $A\in C^0(\ii,\sltwo)\cap C^1(\ii,\mm^2)$ for some interval $\ii$, then $(A,\dot A)\in C^0(\ii,\dd)$.
In particular, $\dot AA^{-1}
\in C^0(\ii,\sllatwo)$.
\end{lemma}

This leads to the following definition.
\begin{definition}
\label{velgrad}
Define the mapping $L:\dd\to\sllatwo$ by
\[
L(A,B)=BA^{-1},
\]
so that the spatial velocity gradient of an affine motion 
$x(t,y)=A(t)y$ is given by $\nabla u(t,x)=L(A(t),\dot A(t))$.
\end{definition}

Lemma \ref{phasespace} suggests that the tangent bundle $\dd$ is the natural phase space for affine
incompressible motion.

Consider a solution of \eqref{pde1}, \eqref{pde2}.
Let us    assume that the velocity  $u(t,x)$ and the fluid domains $\omt$ arise from
an incompressible affine motion $x(t,y)=A(t)y$, as described above.
By Lemma \ref{phasespace}, the velocity field 
is divergence free:
\[
\nabla\cdot u(t,x)=\tr \dot A(t) A(t)^{-1}=0,\quad t\in\rr.
\]

Let us assume that
\[
b(t,x)=\beta(t)A(t)^{-1}x\isp{and} -\nabla p(t,x)=A(t)^{-\top}\varpi(t)A(t)^{-1}x, 
\]
with $\beta, \varpi\in C^2(\rr,\mm^2)$.  We use the
notation $A^{-\top}=(A^{-1})^\top$.
(As motivation, note that if we  assume that  the  unknowns $b$ and $p$ are also spatially
homogeneous, then  the PDEs imply that $\nabla p(t,x)$ and $b(t,x)$ should be
homogeneous of degree one in the variable $x$.)

The equation for the magnetic field implies that
\[
\dot\beta=\dot AA^{-1}\beta
\]
from which it follows that
\[
\beta(t)=A(t)A_0^{-1}\beta_0, \isp{where} A_0=A(0), \quad \beta_0=\beta(0).
\]

The normal vector to $\partial\omt$ at a point $x\in\partial\omt$ has the direction of the vector
$A(t)^{-\top}A(t)^{-1}x$.
So the boundary condition implies that for all $|y|=1$,
\[
0=b(t,x(t,y))\cdot n(t,x(t,y))= \beta(t)y\cdot A(t)^{-\top}y=A_0^{-1}\beta_0y\cdot y.
\]
It follows that $A_0^{-1}\beta_0$ is anti-symmetric, and so there exists a constant $c_0$ such that
\[
A_0^{-1}\beta_0=c_0Z.
\]
Thus, we have shown  that
\[
b(t,x)=c_0A(t)ZA(t)^{-1}x.
\]
As a consequence,
\[
\nabla \cdot b(t,x)=c_0\tr A(t)ZA(t)^{-1}=c_0\tr Z=0,
\]
so that $b$ is divergence free.

Since $A(t)^{-\top}\varpi(t)A(t)^{-1}x$ is a gradient, the matrix $A(t)^{-\top}\varpi(t)A(t)^{-1}$ must be symmetric.
Thus,
\[
\nabla p(t,x)=-\tfrac12\nabla [A(t)^{-\top}\varpi(t)A(t)^{-1}x\cdot x].
\]
We find that
\[
p(t,x)=\tfrac12\left[\lambda(t)-{\varpi(t)A(t)^{-1}x}\cdot{A(t)^{-1}x}\right],
\]
for some scalar function $\lambda(t)$.  The other boundary condition implies that
\[
0=p(t,x(t,y))=
\tfrac12[\lambda(t)-\varpi(t)y\cdot y],
\]
for all $|y|=1$.  This forces
\[
\varpi(t)=\lambda(t) I,
\]
and so
\[
p(t,x)=\tfrac12\lambda(t)[1-|A(t)^{-1}x|^2].
\]

Finally, from the velocity equation, we derive
\[
\ddot A(t)=\lambda(t)A(t)^{-\top}+A(t)(c_0Z)^2=\lambda(t)A(t)^{-\top}-c_0^2A(t).
\]
Since $A\in C^0(\rr,\sltwo)$, we may write $A(t)^{-\top}=\cof A(t)$.
Thus, we have proven

\begin{lemma}
\label{affinehommhd}
Suppose that 
\[
 A\in C^0(\rr,\sltwo)\cap C^2(\rr,\mm^2),\quad   \lambda\in C^0(\rr,\rr),\quad c_0\in\rr,\isp{and} \kk=c_0^2.
\]
Define
\begin{align*}
&u(t,x)=\dot A(t)A(t)^{-1}x,\\
&b(t,x)=c_0 A(t)ZA(t)^{-1}x,\\
&p(t,x)=\tfrac12\lambda(t)\left[1-|A(t)^{-1}x|^2\right],\\
\intertext{and}
&\omt=A(t)\bb.
\end{align*}
Then $u(t,x)$, $b(t,x)$, $p(t,x)$ solve the MHD system \eqref{pde1}, \eqref{pde2} in $\omt$ if and only if
\begin{equation}
\label{mainode}
\ddot A(t)+\kk A(t)=\lambda(t) \cof A(t).
\end{equation}
\end{lemma}

\begin{remark}
\label{pressuresignremark}
Note that  $1-|A(t)^{-1}x|^2>0$, for $x\in\omt$, and so the sign of $ p(t,x)$ is determined by the sign of $\lambda(t)$.
Thus, an affine solution satisfies the Taylor sign condition if and only if $\lambda(t)>0$.
We shall see later on (in Corollary \ref{pressuresign}) that the sign of this function is preserved
under the motion.
\end{remark}

\begin{remark}
\label{lagrdef}
The equations of motion \eqref{mainode} are the Euler-Lagrange equations associated to the Lagrangian
$
\Ell_0:\mm^2\times\mm^2\times\rr\to\rr
$
given by
\[
{\Ell_0}(A,\dot A,\lambda)=\tfrac12|\dot A|^2-\khalf|A|^2+\lambda(\det A-1).
\]
The scalar function $\lambda(t)$ in \eqref{mainode} is a Lagrange multiplier which will now be identified.
\end{remark}

\begin{definition}
\label{lamdef}
Given a  parameter value $\kk\ge0$, define the  Lagrange multiplier map
$\Lambda:\dd\to\rr$
by
\[
\Lambda(A,B)=\frac{2(\kk-\det B)}{|A|^2}.
\]
\end{definition}

The dependence of $\Lambda$ on $\kk$ will  be suppressed.

\begin{lemma}
\label{litlam}
Fix $\kk\ge0$.
Suppose that $A\in C^2(\ii,\mm^2)$ satisfies \eqref{mainode} on some interval $\ii\subset\rr$.
Then $A\in C^0(\ii,\sltwo)$ if and only if $(A(t_0),\dot A(t_0))\in\dd$ for some $t_0\in\ii$
and 
\begin{equation}
\label{LagMultID}
\lambda(t)=\Lambda(A(t),\dot A(t)),
\isp{for} t\in\ii.
\end{equation}
\end{lemma}

\begin{proof}
Suppose that $A\in C^2(\ii,\mm^2)$ satisfies \eqref{mainode}.
By Lemma \ref{detcof}, we have
\[ 
J(t)\equiv \det  A(t)=\half \ip{\cof A(t)}{ A(t)}.
\]
Recall that the map $\cof:\mm^2\to\mm^2$ is linear, symmetric, and orthogonal, by Lemma \ref{cofiso}.  Using this
and Lemma \ref{detcof}, we have
\begin{align}
\label{detode1}
\dot J & = \ip{\cof A}{\dot A}\\
\intertext{and}
\nonumber
\ddot J
& = \ip{\cof A}{\ddot A}+\ip{\cof\dot A}{\dot A}\\
\label{detode2}
&=\ip{\cof A}{-\kk A+\lambda\cof A} +2\det\dot A\\
\nonumber
&=-2\kk J+\lambda|A|^2+2\det\dot A.
\end{align}

If $A\in C^0(\ii,\sltwo)$, then $J(t)=1$, $t\in\ii$, and by \eqref{detode2}, we obtain \eqref{LagMultID}.
Moreover, by Lemma \ref{phasespace}, we have $(A(t),\dot A(t))\in \dd$, for all $t\in\ii$.

On the other hand, if \eqref{LagMultID} holds, then we see that 
\[
\ddot J(t)+2\kk J(t)=0,\quad t\in\ii.
\]
If $(A(t_0),\dot A(t_0))\in\dd$, then $(J(t_0),\dot J(t_0))=(1,0)$ by \eqref{detode1}, and so by uniqueness,
we see that $J(t)=1$, $t\in\ii$.  Thus, $A\in C^0(\rr,\sltwo)$.
\end{proof}

\begin{remark}
Since this result depends on the linearity of the cofactor map on $\mm^2$, it does not
carry over to incompressible flows in 3d.  A dimension independent formula for $\lambda$
was given in \cite{sideris-2017}.
\end{remark}

\section{Existence of global solutions}

\begin{definition}
\label{exdef}
Fix $\kk\ge0$.  For $i=1, 2, 3$,  we define the maps $X_i:\mm^2\times\mm^2\to\rr$ by
\begin{align*}
&X_1(A,B)=\half|B|^2 +\khalf|A|^2 \\
&X_2(A,B)=\half\ip{ZA-AZ}{B} \\
&X_3(A,B)=\half\ip{ZA+AZ}{B} .
\end{align*}
\end{definition}

We suppress the dependence of $X_1$ on the fixed parameter $\kk\ge0$.

\begin{theorem}
\label{inv1}
If $A\in C^0(\ii,\sltwo)\cap C^2(\ii,\mm^2)$ solves \eqref{mainode} on an interval $\ii$,
then the quantities $X_i(A(t),\dot A(t))$,  $i=1, 2, 3$, are invariant.
\end{theorem}

\begin{proof}
Suppose that $A\in C^0(\ii,\sltwo)\cap C^2(\ii,\mm^2)$ solves \eqref{mainode} on $\ii$.
By Lemma \ref{phasespace}, $(A,\dot A)\in C^0(\ii,\dd)$, so that $\ip {\dot A(t)}{\cof A(t)}=0$ on $\ii$.  From this
we obtain
\[
\frac{d}{dt}X_1(A,\dot A)=\ip{\ddot A+\kk A}{\dot A}=\lambda\ip{\cof A}{\dot A}=0,
\]
which proves the invariance of $X_1(A,\dot A)$.  Next we compute
\[
\frac{d}{dt}\ip{ZA}{\dot A}=\ip{Z\dot A}{\dot A}+\ip{ZA}{-\kk A+\lambda \cof A}.
\]
Each term on the right-hand side vanishes because $\ip{Z\widetilde A}{\widetilde A}=0$ for any 
$\widetilde A\in\mm^2$.  To see that the last term vanishes in this way we  use Lemmas \ref{cofiso} and \ref{symip}.
As a result, the quantity $\ip{ZA}{\dot A}$ is invariant.  The same holds for $\ip{AZ}{\dot A}$.
Therefore, the  statements for $i=2,3$ are now clear.
\end{proof}

\begin{remark}
The invariants $X_2\pm X_3$ correspond to the invariance of the  Lagrangian
$\Ell_{0}$ under the left and right actions of $\sotwo$.
\end{remark}

\begin{remark}
Note that by Lemma \ref{so2.1}, $X_1(A,B)\ge \tfrac\kk2|A|^2\ge\kk$, for all $(A,B)\in\dd$.
\end{remark}

A  version of the next result (in 3d and with  $\kk=0$) appeared  in Lemma 4 of \cite{sideris-2017}.

\begin{theorem}
\label{ivp}
Given a  parameter value $\kk\ge0$ and initial data
$(A_0,B_0)\in\dd$,
the initial value problem
\begin{equation}
\label{odes}
\begin{aligned}
&\ddot A(t)+\kk A(t)=\Lambda(A(t),\dot A(t))\cof A(t),\\
&(A(0),\dot A(0))=(A_0,B_0)
\end{aligned}
\end{equation}
has a unique global solution $A\in C^0(\rr,\sltwo)\cap C^2(\rr,\mm^2)$.  Additionally, there holds
$(A,\dot A)\in C^0(\rr,\dd)$.
\end{theorem}

\begin{proof}
Express the equation as an equivalent  first order system
\begin{align*}
&\dot A=B\\
&\dot B=-\kk A+\Lambda(A,B)\cof A.
\end{align*}
The vector field on the right-hand side is a smooth function of $(A,B)\in\mm^2\times\mm^2$
provided $A\ne0$, by Lemma \ref{detder}.  Since the initial data satisfies $A_0\in\sltwo$, we have $A_0\ne 0$.  Therefore,
by the Picard existence and uniqueness
theorem, the problem has a unique solution $(A,B)\in C^1(\ii,\mm^2)$ defined on some
maximal interval $\ii$, and since $\dot A=B$, we have $A\in C^2(\ii,\mm^2)$.  
So by Lemma \ref{litlam}, the solution satisfies $A\in C^0(\ii,\sltwo)$,
and by Lemma \ref{phasespace}, $(A,\dot A)\in C^0(\ii,\dd)$.
By Lemma \ref{so2.1}, $|A(t)|^2\ge2$, on $\ii$, and so $|A(t)|$ is bounded away from 0.
By Theorem \ref{inv1}, the quantity $X_1(A(t),B(t))$ is conserved.  If $\kk>0$, then the norm of the solution is
uniformly bounded in $\mm^2\times\mm^2$, and we conclude that the solution is global.  If $\kk=0$, then the norm of
$B$ is uniformly bounded in $\mm^2$.  But since $\dot A=B$, we see that the norm of $A$ 
can not blow up in finite time, so again we conclude that the solution is global.
\end{proof}

\begin{corollary}
A curve $A\in C^0(\rr,\sltwo)\cap C^2(\rr,\mm^2)$ is a geodesic in $\sltwo$ with
the (induced) Euclidean metric if and only if it satisfies 
\eqref{mainode} with $\kk=0$.
\end{corollary}

\begin{remark}
We include constant solutions as geodesics.
\end{remark}

\begin{proof}
A geodesic curve is one for which $\dot A(t)$ is parallel along $A(t)$.  That is
\[
\frac{D_A}{dt}\dot A(t)=0,
\]
in which the covariant derivative of $\dot A(t)$ along $A(t)$ is
is the projection of $\ddot A(t)$ onto $T_{A(t)}\sltwo$.  Thus, the curve $A\in C^0(\rr,\sltwo)\cap C^2(\rr,\mm^2)$ 
is a geodesic if and only if
$
\ddot A(t)\in\lspan{\cof A(t)}.
$
This is equivalent to \eqref{mainode}.
\end{proof}

\section{Hamiltonian formalism}
\label{HamRed}

We now proceed to reduce the constrained problem for \eqref{mainode} first to an unconstrained Lagrangian
system using the map $\Phi$ and then to a completely integrable Hamiltonian system 
using the Legendre transformation.

\begin{definition}

Given $\kk\ge0$, define a  Lagrangian density $\Ell:\rr^3\times\rr^3\to\rr$ by
\[
\Ell(x,y)=\half|D_x\varphi(x)y|^2-\khalf|\varphi(x)|^2
=\half\ip{g(x)y}{y}-\khalf|\varphi(x)|^2.
\]
\end{definition}

\begin{lemma}
\label{thelagodelemma}
If $x\in C^2(\ii,\rr^3)$ is a solution of the Lagrangian equations of motion for $\Ell$:
\begin{equation}
\label{thelagode}
\frac{d}{dt}D_y\Ell(x,\dot x)-D_x\Ell(x,\dot x)=0,
\end{equation}
on some time interval $\ii$, then $A=\varphi\circ x\in C^0(\ii,\sltwo)\cap C^2(\ii,\mm^2)$ is a solution of \eqref{mainode}.
\end{lemma}

\begin{proof}

The system \eqref{thelagode} implies 
\[
D_x\varphi\circ x^\top[\ddot A+\kk A]=0,
\]
or in other words, by \eqref{kernel}
\[
\ddot A+\kk A\in\ker D_x\varphi\circ x^\top = (\image D_x\varphi\circ x)^\perp=\lspan \cof\varphi\circ x=\lspan\cof A.
\]
This shows that $A$ satisfies \eqref{mainode}.

\end{proof}

\begin{lemma}
\label{Ham1}
The Hamiltonian\footnote{In this section, $p\in\rr^3$ denotes the generalized momentum, and not
the pressure.}
 associated to $\Ell$ is
\[
H(x,p)=\half\ip{g^{-1}(x)p}{p}+\khalf|\varphi(x)|^2,\quad (x,p)\in\rr^3\times\rr^3,
\]
with
\begin{equation}
\label{invmetric}
g^{-1}(x)=
\begin{bmatrix}
1-x_1^2/|\varphi(x)|^2 & -x_1x_2/|\varphi(x)|^2 & 0\\
-x_1x_2/|\varphi(x)|^2 & 1-x_2^2/|\varphi(x)|^2 & 0\\
0 & 0 & 1/\rho(x)^2
\end{bmatrix}.
\end{equation}
\end{lemma}
\begin{proof}
The Legendre transformation associated to $\Ell$ is
\[
P(x,y)=D_y\Ell(x,y)=g(x)y,\quad (x,y)\in\rr^3\times\rr^3.
\]
The inverse transformation is 
\[
Y(x,p)=g^{-1}(x)p,\quad 
(x,p)\in\rr^3\times\rr^3,
\]
 and the Hamiltonian associated to $\Ell$ is
\[
H(x,p)=\ip{Y(x,p)}{p}-\Ell(x,Y(x,p))=\half\ip{g^{-1}(x)p}{p}+\khalf|\varphi(x)|^2.
\]
The formula \eqref{invmetric} is a direct computation from \eqref{metric}.
\end{proof}

Here we can think of $(x,p)\in \rr^3\times\rr^3$ as a point in the cotangent bundle of $\rr^3$,
with the Legendre transformation $y\mapsto p=g(x)y$ taking $T_x\rr^3$ to $T^\ast_x\rr^3$.

Using the facts that
\[
\rho(x)^2=2+|\bx|^2,\quad |\varphi(x)|^2=\rho(x)^2+|\bx|^2,\quad |\bx|^2|\bp|^2=\ip{\bx}{\bp}^2+\ip{Z\bx}{\bp}^2,
\]
we can write
\[
\ip{g^{-1}(x)p}{p}=|\bp|^2-\frac{\ip{\bx}{\bp}^2}{|\varphi(x)|^2}+\frac{p_3^2}{\rho(x)^2}
=\frac{\rho(x)^2|\bp|^2+\ip{Z\bx}{\bp}^2}{|\varphi(x)|^2}+\frac{p_3^2}{\rho(x)^2},
\]
and thus,
\begin{equation*}
\label{mainham}
H(x,p)=\half\left(\frac{\rho(x)^2|\bp|^2+\ip{Z\bx}{\bp}^2}{|\varphi(x)|^2}+\frac{p_3^2}{\rho(x)^2}\right)
+\khalf|\varphi(x)|^2.
\end{equation*}
The   Hamiltonian system
\[
\dot x = D_pH(x,p),\quad \dot p=-D_xH(x,p)
\]
takes the explicit form
\begin{equation}
\label{Hamsys}
\begin{aligned}
&\dot \bx=\frac{\rho(x)^2}{|\varphi(x)|^2}\;\bp+\frac{\ip{Z\bx}{\bp}}{|\varphi(x)|^2}\;Z\bx\\
&\dot x_3=\frac{p_3}{\rho(x)^2}\\
&\dot\bp=
\left(\frac{2|\bp|^2+2\ip{Z\bx}{\bp}^2}{|\varphi(x)|^4}+\frac{p_3^2}{\rho(x)^4}-2\kk\right)\bx
+\frac{\ip{Z\bx}{\bp}}{|\varphi(x)|^2}\;Z\bp\\
&\dot p_3=0.
\end{aligned}
\end{equation}

\begin{lemma}
\label{ham2main}
If $(x,p)\in C^1(\ii,\rr^3\times\rr^3)$ is a solution of \eqref{Hamsys} on $\ii$,
then
$A=\varphi\circ x$ is a solution of \eqref{mainode}
in $C^0(\ii,\sltwo)\cap C^2(\ii,\mm^2)$.
\end{lemma}

\begin{proof}
Since $H$ is smooth, it follows from \eqref{Hamsys} that $x\in C^2(\ii,\rr^3)$, and thus,
by Lemma \ref{varphiimmersion}
$\varphi\circ x\in C^2(\ii,\sltwo)$.

That $x($ solves the associated Lagrangian system \eqref{thelagode} is a 
standard fact, and by Lemma \ref{thelagodelemma}, we have that $\varphi\circ x$ solves \eqref{mainode}.
\end{proof}

\begin{lemma}
\label{gammacoord}
The map
\[
\Gamma(x,p)=(x,g^{-1}(x)p)
\]
is a bijection on $\rr^3\times\rr^3$, and the map $\Phi\circ\Gamma:\rr^3\times\rr^3\to\dd$ 
is  surjective.

\end{lemma}

\begin{proof}
The first statement is obvious, and the second follows from Lemma \ref{bigphionto}.
\end{proof}

\begin{lemma}
\label{gammainv}
If $(A,B)=\Phi\circ\Gamma(x,p)$, for some 
 $(x,p)\in\rr^3\times\rr^3$,  then
\begin{align*}
&X_1(A,B)=H(x,p),\\
&X_2(A,B)=x_1p_2-x_2p_1=\ip{Z\bx}{\bp},\\
&X_3(A,B)=p_3.
\end{align*}
\end{lemma}

\begin{proof}
This lemma is a direct computation based on the definitions of $X_i$, $\Phi$, and $\Gamma$.
\end{proof}

The 
appearance the invariants $X_2$ and $X_3$ in \eqref{Hamsys} will eventually allow us to uncouple
the system.

\begin{theorem}
\label{HamInv}
If $(x,p)\in C^1(\ii,\rr^3\times\rr^3)$ is a solution of \eqref{Hamsys} on $\ii$,
then the quantities
\[\
X_i\circ\Phi\circ\Gamma(x(t),p(t)),\quad i=1, 2, 3
\]
are invariant.
The  three invariants Poisson commute, and so the system \eqref{Hamsys} is completely integrable.
\end{theorem}

\begin{proof}
By Lemma \ref{ham2main}, we know that $A=\varphi\circ x$ is a solution of \eqref{mainode}
in $C^0(\ii,\sltwo)\cap C^2(\ii,\mm^2)$.
Since $\dot{x} = D_pH(x,p)=g^{-1}(x)p$, we see that
\[
\Phi\circ\Gamma(x,p)=(\varphi(x),D_x\varphi(x)g^{-1}(x)p)=(\varphi(x),D_x\varphi(x)\dot x)=(A,\dot A).
\]
This implies that
\[
X_i\circ\Phi\circ\Gamma(x,p)=X_i(A,\dot A),
\]
so it follows from Lemma \ref{inv1} that these quantities are conserved.
\end{proof}

\begin{theorem}
For any $\kk\ge0$ and any initial data $(x(0),p(0))\in\rr^3\times\rr^3$, the system \eqref{Hamsys}
has global solutions.

\end{theorem}

\begin{proof}
By Theorem \ref{HamInv}, $X_1(x(t),p(t))$ is conserved along any solution $(x(t),p(t))$.
If $\kk>0$, this implies that all solutions remain bounded, and thus, global existence holds.
If $\kk=0$, then $|p(t)|$ is bounded.  It then follows from \eqref{Hamsys} that $|\dot x(t)|$
is uniformly bounded, which prevents blow up in finite time.
\end{proof}

\begin{lemma}
\label{so2x2}
If $(A,B)\in\dd_0$, then $X_2(A,B)= 0$.
If $(A,B)\in\dd$ and $X_2(A,B)\ne 0$, then $A\in\sltwo\setminus\sotwo$.
\end{lemma}

\begin{proof}
If $(A,B)\in\dd_0$, then $A\in\sotwo$. 
By Lemma \ref{so2.2}, $ZA-AZ=0$, and so we have $X_2(A,B)=0$.
If $(A,B)\in\dd$ and $X_2(A,B)\ne 0$, then $(A,B)\in\dd\setminus\dd_0$, so
$A\in\sltwo\setminus\sotwo$.
\end{proof}

\begin{corollary}
\label{passthrough}
Let $(x,p)\in C^1(\rr,\rr^3\times\rr^3)$ be a (necessarily global) solution of \eqref{Hamsys},
and set  
\[
X_2(t)=x_1(t)p_2(t)-x_2(t)p_1(t).
\]
 If $X_2(t_0)\ne0$ for a single time $t_0\in\rr$, then $\bx(t)\ne0$ and 
 $\varphi\circ x(t)\in \sltwo\setminus\sotwo$,  for all $t\in\rr$.
If  $\bx(t_0)=0$
 for a single time $t_0\in \rr$, then $X_2(t)=0$, for all $t\in\rr$.
\end{corollary}

\begin{proof}
By Theorem \ref{HamInv}, the quantity $X_2(t)$ is invariant.  Therefore,
either $X_2(t)=0$ for all $t\in\rr$,  or $X_2(t)\ne0$, for all $t\in\rr$.
So if $X_2(t)$ vanishes for a single time, then it vanishes identically.
Otherwise, $X_2(t)\ne0$ on $\rr$ which forces $\bx(t)\ne0$ on $\rr$,
and by Lemma \ref{varphiimmersion},
$\varphi\circ x(t)\in\sltwo\setminus\sotwo$ on $\rr$.

  \end{proof}

The behavior  of the system \eqref{Hamsys} is topologically different depending on whether the invariant
$X_2$ is nonzero or not.  We shall now consider each case in turn.

\section{Reduction to the phase plane in the case $X_2\ne0$}
\label{Hamkp}

According to Corollary \ref{passthrough}, the property $\bx(t_0)\ne0$ is preserved
by the  flow of \eqref{Hamsys}.  Therefore, it is natural to introduce polar coordinates
for $\bx$ in this case.

\begin{definition}
\label{psidef}
Define the map $\psi:\rr^3\to\rr^3$ by
\[
\psi(q)=
\begin{bmatrix}
q_1\cos q_2\\ q_1\sin q_2\\ q_3
\end{bmatrix}.
\]
\end{definition}

\begin{remark}
\label{halfspace}
We shall mostly consider $\psi$ on the restricted domain $\rr^3_+=\{q\in\rr^3:q_1>0\}$.
\end{remark}

\begin{lemma}
\label{phipsiimmersion}
The map $\varphi\circ\psi:\rr^3\to\sltwo$ is surjective, and $\varphi\circ\psi(q)\in\sotwo$ if and only if $q_1=0$.

The map $\varphi\circ\psi$ restricted to $\rr^3_+$ is an immersion from $\rr^3_+$ onto $\sltwo\setminus\sotwo$.
The restricted map $\varphi\circ\psi $ defines a local coordinate chart for $\sltwo\setminus\sotwo$ 
in a neighborhood of any point $q\in\rr^3_+$.

\end{lemma}

\begin{proof}
The map $\psi:\rr^3\to\rr^3$ is surjective, and  $\psi_1(q)=\psi_2(q)=0$ if and
only if $q_1=0$, so the first statement follows by Lemma \ref{varphiimmersion}.

For each $q\in\rr^3_+$, $D_q\psi(q)$ is a bijection on $\rr^3$, and so by Lemma \ref{varphiimmersion},
\[
D_q[\varphi\circ\psi(q)]:\rr^3\to T_{\varphi\circ\psi(q)}\sltwo
\]
is a bijection for each $q\in\rr^3_+$.  This shows that $\varphi\circ\psi$ is an immersion,
and it follows that $\varphi\circ\psi$ defines a local coordinate chart near any point $q\in\rr^3_+$.

\end{proof}

\begin{lemma}
\label{Aform1}
If $A\in\sltwo$ 
and 
$
A=\varphi\circ\psi(q),
$
for some $q\in\rr^3$, then  $\half|A|^2=1+q_1^2$ and
\[
A=\tfrac1{\sqrt2}U \left(\half(q_3+q_2)\right) \left(\rho I+q_1 K\right)U \left(\half(q_3-q_2)\right),
\isp{with}\rho=(2+q_1^2)^{1/2}.
\]
The image of the unit disk
under $A$ is an ellipse with principal axes of lengths
\begin{equation*}
\left[(\half|A|^2+1)^{1/2}\pm(\half|A|^2-1)^{1/2}\right]/\sqrt2.
\end{equation*}

\end{lemma}

\begin{proof}
By definition, we have
\[
A=\varphi\circ\psi(q)
=\tfrac1{\sqrt2}\left[\rho\left(\cos q_3I+\sin q_3 Z\right)+q_1\left(\cos q_2K+\sin q_2 M\right)\right],
\]
with $ \rho=(2+q_1^2)^{1/2}$.
Recalling definition \ref{rotmat}, we set
$U(\theta)=\cos \theta I+\sin\theta Z$.
Since $M=ZK=-KZ$,  we have
\[
\cos\theta K+\sin\theta M= (\cos\theta I+\sin\theta Z)K=U(\theta)K=KU(-\theta).
\]
Therefore, we obtain
\begin{align*}
\sqrt2A
&=\rho\; U\left(q_3\right)+q_1U\left(q_2\right)K\\
&=U\left(\half(q_3+q_2)\right)\rho I\;U \left(\half(q_3-q_2)\right)+U \left(\half(q_3+q_2)\right)q_1 KU \left(\half(q_3-q_2)\right)\\
&=U \left(\half(q_3+q_2)\right) \left(\rho I+q_1 K\right)U \left(\half(q_3-q_2)\right).
\end{align*}
Since $\rho I+q_1K=\diag \begin{bmatrix}
\rho+q_1 & \rho-q_1
\end{bmatrix}$, the formula shows that the image of the unit disk
under $A$ is an ellipse with principal axes of lengths 
$
(\rho\pm q_1)/\sqrt2,
$
giving \eqref{praxes}.
\end{proof}

\begin{definition}
\label{cappsidef}
Set $\rr^1_3=\{(0,0,q_3)\in\rr^3:q_3\in\rr\}$, and define the map 
$
\Psi:\rr^3_+\times\rr^3\to(\rr^3\setminus\rr^1_3)\times\rr^3
$
 by
\[
\Psi(q,\xi)=(\psi(q),D\psi(q)^{-\top}\xi).
\]
\end{definition}

$\Psi$ is well-defined since $D\psi(q)$ is invertible when $q\in\rr^3_+$.

\begin{lemma}
\label{canonical}
The transformation $\Psi$ is canonical. 
\end{lemma}

\begin{lemma}
\label{phigammapsi}
The composition
$\Phi\circ\Gamma\circ\Psi$ maps $\rr^3_+\times\rr^3$ onto $\dd\setminus\dd_0$.
The induced metric  on $T_{\varphi\circ\psi(q)}\sltwo$ is
\[
h(q)=\diag
{\everymath{\displaystyle}
\begin{bmatrix}
\frac{2+q_1^2}{2(1+q_1^2)}&\frac1{q_1^2} & \frac1{2+q_1^2}
\end{bmatrix}.}
\]
\end{lemma}

\begin{proof}
The first statement follows from Lemmas \ref{bigphionto}, \ref{gammacoord}, and \ref{phipsiimmersion}.
Explicitly, we have
\[
\Phi\circ\Gamma\circ\Psi(q,\xi)=(\varphi\circ\psi(q),T(q)\xi)
\]
with 
\[
T(q)\equiv D_x\varphi\circ\psi(q) \;g^{-1}\circ\psi(q)\;D_q\psi(q)^{-1}:\rr_+^3\to T_{\varphi\circ\psi(q)}\sltwo.
\]
The metric $h(q)$ can be found by computing 
\[
h(q)=T(q)^\top T(q) =D_q\psi^{-1}(q)\;g^{-1}\circ\psi(q)\;D_q\psi(q)^{-\top}.
\]
\end{proof}

\begin{definition}
\label{Ham2}
Define the Hamiltonian
\begin{multline*}
\widetilde H(q,\xi)=H\circ\Psi(q,\xi)
=\half\ip{h(q)\xi}{\xi}+\khalf|\varphi\circ\psi(q)|^2\\
=\half\left(\frac{(2+q_1^2)\xi_1^2}{2(1+q_1^2)}+\frac{\xi_2^2}{q_1^2}+\frac{\xi_3^2}{2+q_1^2}\right)+\kk(1+q_1^2),
\end{multline*}
for $ (q,\xi)\in\rr^3_+\times\rr^3$.
\end{definition}

The corresponding system is
\begin{equation}
\begin{aligned}
\label{Hamsys2}
&\dot q_1=\frac{(2+q_1^2)\xi_1}{2(1+q_1^2)}
&\qquad&\dot \xi_1=\left(\frac{\xi_1^2}{2(1+q_1^2)^2}+\frac{\xi_2^2}{q_1^4}+\frac{\xi_3^2}{(2+q_1^2)^2}-2\kk\right) q_1\\
&\dot q_2=\frac{\xi_2}{q_1^2}&&\dot \xi_2=0\\
&\dot q_3=\frac{\xi_3}{2+q_1^2}&&\dot\xi_3=0.
\end{aligned}
\end{equation}
Notice that our choice of polar coordinates for $\bx$ has created a singularity at $q_1=|\bx|=0$,
corresponding to $\sotwo$.

\begin{lemma}
\label{can}
If $(q,\xi)\in C^1(\ii,\rr^3_+\times\rr^3)$ is a solution of \eqref{Hamsys2} on $\ii$, then
$(x,p)=\Psi(q,\xi)$ solves \eqref{Hamsys} on $\ii$.
\end{lemma}

\begin{proof}
The  statement holds because the transformation $\Psi$ is canonical.
\end{proof}

\begin{lemma}
\label{qxiinv1}
If $(A,B)=\Phi\circ\Gamma\circ\Psi(q,\xi)$, for some 
 $(q,\xi)\in\rr^3\times\rr^3$,  then
 \[
 X_1(A,B)=\widetilde H(q,\xi),
 \quad
 X_i(A,B)=\xi_i,
\quad
i=2,3,
\]
and these quantities are invariant along solutions of \eqref{Hamsys2}.
\end{lemma}

\begin{lemma}
\label{lb}
If $(q,\xi)\in C^1(\ii,\rr^3_+\times\rr^3)$ is a solution of \eqref{Hamsys2} on $\ii$
with $\xi_2\ne0$, then  $q_1(t)$ is bounded away from zero:
\[
q_1(t)^2\ge  \xi_2^2/(2\widetilde H(q,\xi))=X_2^2/(2X_1)>0,
\]
and $(q(t),\xi(t))$ is defined for all $t\in\rr$.
\end{lemma}

\begin{proof}
The lower bound follows from Defintion \ref{Ham2} and Lemma \ref{qxiinv1}.
Global existence is now a consequence of the invariance of $X_1$.
\end{proof}

\begin{theorem}
\label{mainthm1}
Fix $\kk\ge0$.
Let $(A_0,B_0)\in\dd\setminus\dd_0$ with $X_2(A_0,B_0)\ne0$.  Choose $(q(0),\xi(0))\in\rr^3_+\times\rr^3$ such that 
$\Phi\circ\Gamma\circ\Psi(q(0),\xi(0))=(A_0,B_0)$.  Let 
$(q,\xi)\in C^1(\rr,\rr^3_+\times\rr^3)$ be the global solution of \eqref{Hamsys2} with initial data $(q(0),\xi(0))$.
  Then 
  \[
A(t)=\varphi\circ\psi\circ q(t), \quad t\in\rr.
\]
is the  solution of \eqref{mainode} in $C^0(\rr,\sltwo)\cap C^2(\rr,\mm^2)$ with initial data $(A_0,B_0)$.
Explicitly, we have
\[
A(t)=\tfrac1{\sqrt2}U \left(\half(q_3(t)+q_2(t))\right) \left(\rho(t) I+q_1(t) K\right)U \left(\half(q_3(t)-q_2(t))\right),
\]
with $\rho(t)=(2+q_1(t)^2)^{1/2}$.  There holds
\[
\half |A(t)|^2=1+q_1(t)^2,
\]
and the fluid domain $\omt=A(t)\bb$ is an ellipse with principal
  axes of lengths
\begin{equation}
\label{praxes}
\tfrac1{\sqrt2}\left[\left(\half |A(t)|^2+1\right)^{1/2}\pm\left(\half |A(t)|^2-1\right)^{1/2}\right],
\end{equation}
for all $t\in\rr$.
\end{theorem}

\begin{proof}
That $A$ is a global solution of \eqref{mainode} follows from Lemmas \ref{lb}, \ref{can},  and \ref{mainham}.
The remaining statements follow from Lemma \ref{Aform1}.
\end{proof}

The Hamiltonian flow \eqref{Hamsys2} is easily understood.
Since $\xi_i=X_i(A_0,B_0)$, $i=2,3$, by Lemma \ref{qxiinv1}, we see from \eqref{Hamsys2} 
that the equations for the pair $(q_1,\xi_1)$
uncouple from the others.  The corresponding orbits are simply the level curves
of $\widetilde H$ in the half plane $\{(q_1,\xi_1):q_1>0\}$ with the known fixed values of $(\xi_2,\xi_3)$.
These level curves are described in  Theorem \ref{htlevelsets} and illustrated in Figures \ref{phasediagram1} and \ref{phasediagram2}.
Given $q_1$, the other coordinates $q_2, q_3$ are found by integration of the remaining equations
in \eqref{Hamsys2}.

\begin{theorem}
\label{htlevelsets}
 For each fixed $(\xi_2,\xi_3)\in\rr^2$ with $\xi_2\ne0$, 
  $\widetilde H$ is a strictly convex function of $(q_1,\xi_1)$
 on the set $\{(q_1,\xi_1):q_1>0\}$.  Its level sets are symmetric with respect to the $\xi_1$ axis.
 If $\kk>0$, $\widetilde H$ has a minimum value  at a unique point $(q_1(\xi_2,\xi_3),0)$, and all other
 level sets are smooth simple closed curves.  See Fig.\  \ref{phasediagram1}.  
 If $\kk=0$, the level sets of $\widetilde H$ are  smooth curves,
 bounded in $\xi_1$ and unbounded in $q_1$.  See Fig.\  \ref{phasediagram2}.  
 
\end{theorem}

%%%%%%%%%

\begin{figure}[h]
\caption{Typical level curves of $\widetilde H$,  in the case  $\kk>0$.}
 \label{phasediagram1}
\ \\
\setlength\unitlength{1mm}
\begin{center}

\begin{overpic}[scale=.75]%,grid,tics=5]
{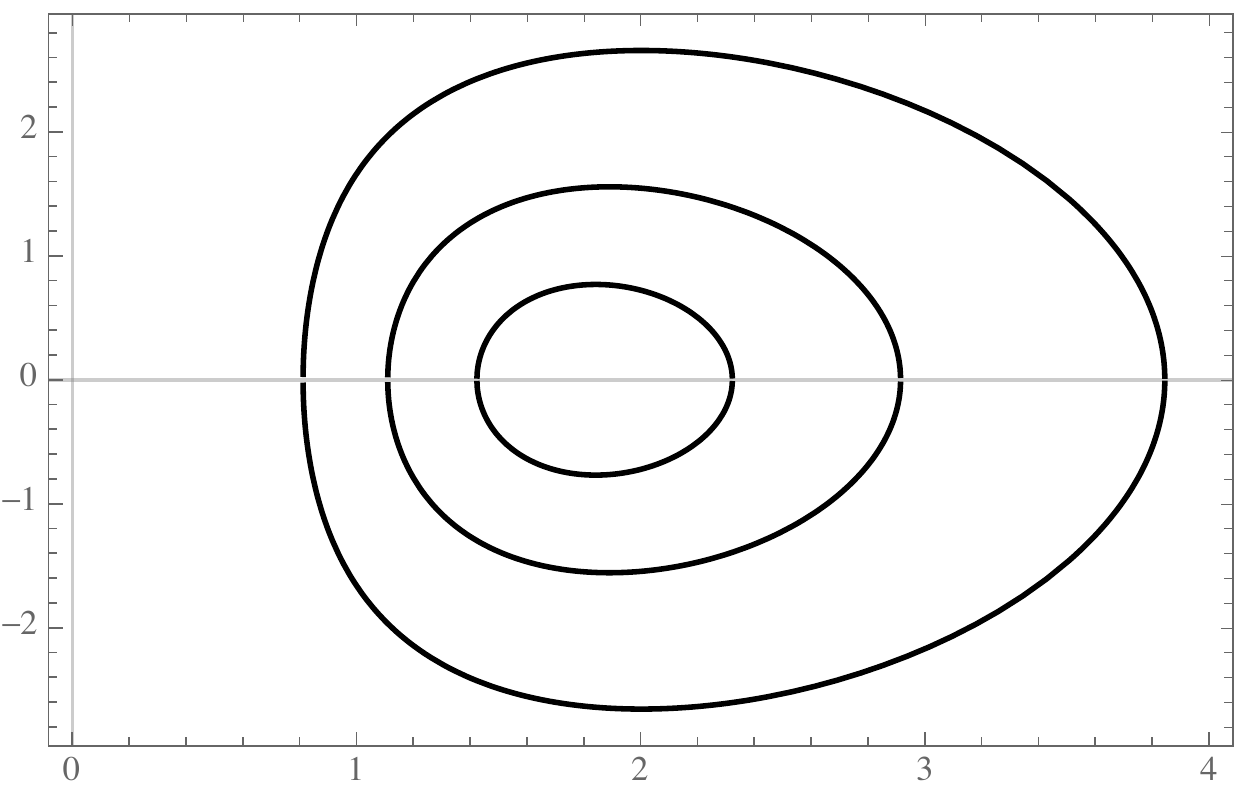}
\put(46,33.2){\circle*{1.5}}
\linethickness{0.5mm}
\put (50,59.6){\vector(1,0){2}}
\put (50,48.6){\vector(1,0){2}}
\put (50,40.6){\vector(10,-1){2}}
\end{overpic}
\end{center}
\end{figure}
%%%%%%%%%%%%%%%%

%%%%%%%%%

\begin{figure}[h]
\caption{Typical level curves of $\widetilde H$,  in the case  $\kk=0$.}
 \label{phasediagram2}
\ \\
\setlength\unitlength{1mm}
\begin{center}

\begin{overpic}[scale=.75]%,grid,tics=5]
{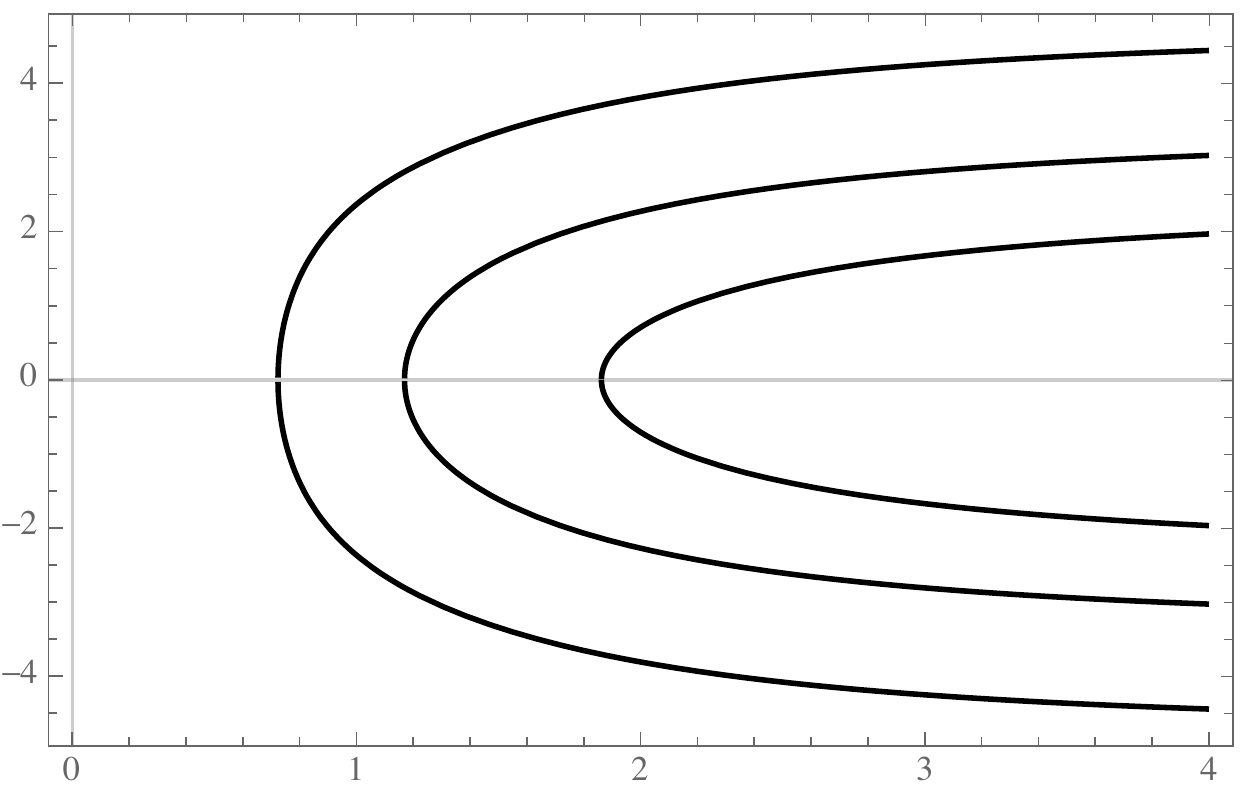}
\linethickness{0.5mm}
\put (70,58){\vector(15,1){2}}
\put (70,49.4){\vector(20,1){2}}
\put (70,42.5){\vector(10,1){2}}
\end{overpic}
\end{center}
\end{figure}
%%%%%%%%%%%%%%%%

\section{Reduction to the phase plane in the case $X_2=0$}
\label{Hamkz}

When $X_2=0$, we shall rely on the fact that for any fixed unit vector $\bar v\in\rr^2$, the
set 
\[
\{(x,p)\in\rr^3\times\rr^3:\ip{\bx}{\bar v}=\ip{\bp}{\bar v}=0\}
\]
is invariant under the flow of \eqref{Hamsys}.

\begin{definition}
Define $\Psi_0:\rr^3\times\rr^3\to\rr^3\times\rr^3$ by
$\Psi_0(q,\xi)=(\psi(q),\psi(\xi))$, where $\psi$ was given in Definition \ref{psidef}.
\end{definition}

\begin{definition}
\label{Ham3}
Define the Hamiltonian
\[
 H_0(q,\xi)=\left.H\circ\Psi_0(q,\xi)\right|_{q_2=\xi_2}
=\half\left(\frac{(2+q_1^2)\xi_1^2}{2(1+q_1^2)}+\frac{\xi_3^2}{2+q_1^2}\right)+\kk(1+q_1^2), \quad (q,\xi)\in\rr^3\times\rr^3.
\]
\end{definition}

\begin{lemma}
\label{phigammapsizero}
If $(A,B)\in\dd$ and $X_2(A,B)=0$, then there exists $(q,\xi)\in\rr^3\times\rr^3$
with $q_2=\xi_2$ such that
\[
(A,B)=\Phi\circ\Gamma\circ\Psi_0(q,\xi).
\]
In this case, we have
\[
X_1(A,B)=H_0(q,\xi),\quad X_2(A,B)=0, \isp{and} X_3(A,B)=\xi_3.
\]
\end{lemma}

\begin{proof}
Let $(A,B)\in\dd$ with $X_2(A,B)=0$.  By Lemma \ref{gammacoord}, we can choose
$(x,p)\in\rr^3\times\rr^3$ such that 
\[
\Phi\circ\Gamma(x,p)=(A,B).
\]
By Lemma  \ref{gammainv}, we have $x_1p_2-x_2p_1=0$.  This says that the vectors
$\bx$ and $\bp$ are dependent.  Therefore, we can find a single unit vector
$\bar\omega\in\rr^2$  and scalars $q_1, \xi_1\in\rr$ such that
\[
\bx=q_1\bar\omega\isp{and}\bp=\xi_1\bar\omega.
\]
We can express $ \bar\omega$ in the form
\[
\bar\omega=(\cos q_2, \sin q_2),
\]
for some $q_2\in\rr$.  Now, if we define
\[
q=(q_1,q_2,x_3)\isp{and} \xi=(\xi_1,q_2,p_3),
\]
then $\Psi_0(q,\xi)=(x,p)$ and $q_2=\xi_2$.
It  follows that $(A,B)=\Phi\circ\Gamma\circ\Psi_0(q,\xi)$.

The form of the nonzero invariants follows from Lemma \ref{gammainv}
and Definition \ref{Ham3}.
\end{proof}

Although $H_0$ is independent of $q_2,q_3,\xi_2$, we shall still regard it as a function on $\rr^3\times\rr^3$.
As such, the corresponding Hamiltonian system takes the form
\begin{equation}
\begin{aligned}
\label{Hamsys3}
&\dot q_1=\frac{(2+q_1^2)\xi_1}{2(1+q_1^2)}
&\qquad&\dot \xi_1=\left(\frac{\xi_1^2}{2(1+q_1^2)^2}+\frac{\xi_3^2}{(2+q_1^2)^2}-2\kk \right)q_1\\
&\dot q_2=0&&\dot \xi_2=0\\
&\dot q_3=\frac{\xi_3}{2+q_1^2}&&\dot\xi_3=0.
\end{aligned}
\end{equation}
Notice that \eqref{Hamsys3} is formally obtained from \eqref{Hamsys2} by deleting the terms which are singular at $q_1=0$,
although here we have $q_1\in\rr$ rather than $q_1>0$.

\begin{lemma}
For any  $\kk\ge0$ and any initial data $(q(0),\xi(0))\in\rr^3\times\rr^3$,
the system \eqref{Hamsys3} has a unique global solution $(q,\xi)\in C^1(\rr,\rr^3\times\rr^3)$.
\end{lemma}

\begin{proof}
The follows by the conservation of the Hamiltonian $H_0$ along the flow of \eqref{Hamsys3}.
\end{proof}

\begin{theorem}
\label{mainthm2}
Fix $\kk\ge0$.
Let $(A_0,B_0)\in\dd$ with $X_2(A_0,B_0)=0$.  Choose $(q(0),\xi(0))\in\rr^3\times\rr^3$ with $q_2(0)=\xi_2(0)$
such that 
$\Phi\circ\Gamma\circ\Psi_0(q(0),\xi(0))=(A_0,B_0)$.  Let 
$(q,\xi)\in C^1(\rr,\rr^3\times\rr^3)$ be the global solution of \eqref{Hamsys3} with initial data $(q(0),\xi(0))$.
  Then
\[
A(t)=\varphi\circ\psi\circ q(t), \quad t\in\rr
\]
is  the solution of \eqref{mainode} in $C^0(\rr,\sltwo)\cap C^2(\rr,\mm^2)$ with initial data $(A_0,B_0)$.
Explicitly, we have
\[
A(t)=\tfrac1{\sqrt2}U \left(\half(q_3(t)+q_2(0))\right) \left(\rho(t) I+q_1(t) K\right)U \left(\half(q_3(t)-q_2(0))\right),
\]
with $\rho(t)=(2+q_1(t)^2)^{1/2}$.  There holds
\[
\half |A(t)|^2=1+q_1(t)^2,
\]
and the fluid domain $\omt=A(t)\bb$ is an ellipse with 
principal axes of lengths given in \eqref{praxes}.
\end{theorem}

\begin{proof}
Since $q_2(0)=\xi_2(0)$, the solution $(q,\xi)$ of \eqref{Hamsys3} satisfies 
\[
q_2(t)=q_2(0)=\xi_2(0)=\xi_2(t), \isp{for all} t\in\rr.
\]
Let $\bar\omega= (\cos q_2(0),\sin q_2(0))$.
Define $(x,p)=\Psi_0(q,\xi)$.  Then $(x,p)\in C^1(\rr,\rr^3\times\rr^3)$ and
\begin{align*}
&x(t)=(\bx(t),x_3(t))=(q_1(t)\bar\omega,q_3(t)),\\
& p(t)=(\bp(t),p_3(t))=(\xi_1(t)\bar\omega,\xi_3(t)),
\end{align*}
for all $t\in\rr$.  According to our definitions, we find
\[
\ip{Z\bx}{\bp}=0,\quad \rho(x)^2=2+q_1^2,\quad |\varphi(x)|^2=2(1+q_1^2),
\isp{and} |\bp|^2=\xi_1^2,
\]
where for notational convenience we have suppressed the time argument.

It is now straightforward to verify that $(x,p)$
  is a global solution
of \eqref{Hamsys} with initial data $\Psi_0(q(0),\xi(0))$.
It follows from Lemma \ref{ham2main} that 
\[
A=\varphi\circ x =\varphi\circ\psi\circ q\in C(\rr,\sltwo)\cap C^2(\rr,\mm^2)
\]
is a solution of \eqref{mainode}.

The remaining statements follow from Lemma \ref{Aform1}.
\end{proof}

 Again, we can understand the flow of \eqref{Hamsys3} by looking at the level sets of $H_0$ with $\xi_3$ fixed.
 These are described in  Theorem \ref{hzlevelsets}, 
 and illustrated in Figures \ref{phasediagram3}, \ref{phasediagram4}, \ref{phasediagram5},
 and \ref{phasediagram6}.

  \begin{theorem}
  \label{hzlevelsets}
Fix $\kk\ge0$ and $\xi_3\in\rr$.  As a function of $(q_1,\xi_1)\in\rr^2$, 
the level curves of $H_0$  are symmetric with respect to the coordinate axes.
They have the following structure.

If $0<\kk<\tfrac18 \xi_3^2$, then $H_0$ has three nondegenerate critical points:
a saddle point at $(0,0)$ and two centers 
 $(\pm q_1(\xi_3),0)$
 where $H_0$ achieves its minimum value.
Each level set  when $\min H_0<H_0<H_0(0,0)$ 
corresponds to a pair of closed orbits.
Each level set when $H_0>H_0(0,0)$ 
corresponds to single closed orbit.
The level set when $H_0=H_0(0,0)$ corresponds to the union of an equilibrium point at $(0,0)$ and two homoclinic orbits.
See Fig.\  \ref{phasediagram3}.

If $\kk>0$ and $\kk\ge\tfrac18\xi_3^2$, then $H_0$ has a single critical point at $(0,0)$ where
it achieves its minimum value.  Each level set  when $H_0>\min H_0$
 corresponds to a closed orbit.  See Fig.~\ref{phasediagram4}.
 
 If $\kk=0$ and $\xi_3\ne0$, then $H_0$ has a saddle point at $(0,0)$.  
 The level set when $H=H_0(0,0)$ corresponds to an equilibrium solution and
 four  orbits, bounded in $\xi_1$ and unbounded in $q_1$, comprising the stable and unstable manifolds
 of the eqilibrium.
 Each level set for $H_0\ne H_0(0,0)$ corresponds to a  pair of orbits, bounded in $\xi_1$ and unbounded in $q_1$.
 See Fig.\  \ref{phasediagram5}.
 
 If $\kk=0$ and $\xi_3=0$, then $\min H_0=0$, and the minimum level set
 corresponds to a line of equilibria $\xi_1=0$.  Each level set for $H_0>0$
 corresponds to a  pair of orbits, bounded in $\xi_1$ and unbounded in $q_1$.
 See Fig.\  \ref{phasediagram6}.
\end{theorem}

%%%%%%%%%

\begin{figure}[h]
\caption{Typical level curves of $ H_0$,  in the case  $0<\kk<\tfrac18X_3^2$.}
 \label{phasediagram3}
\ \\
\setlength\unitlength{1mm}
\begin{center}

\begin{overpic}[scale=.75]%,grid,tics=5]
{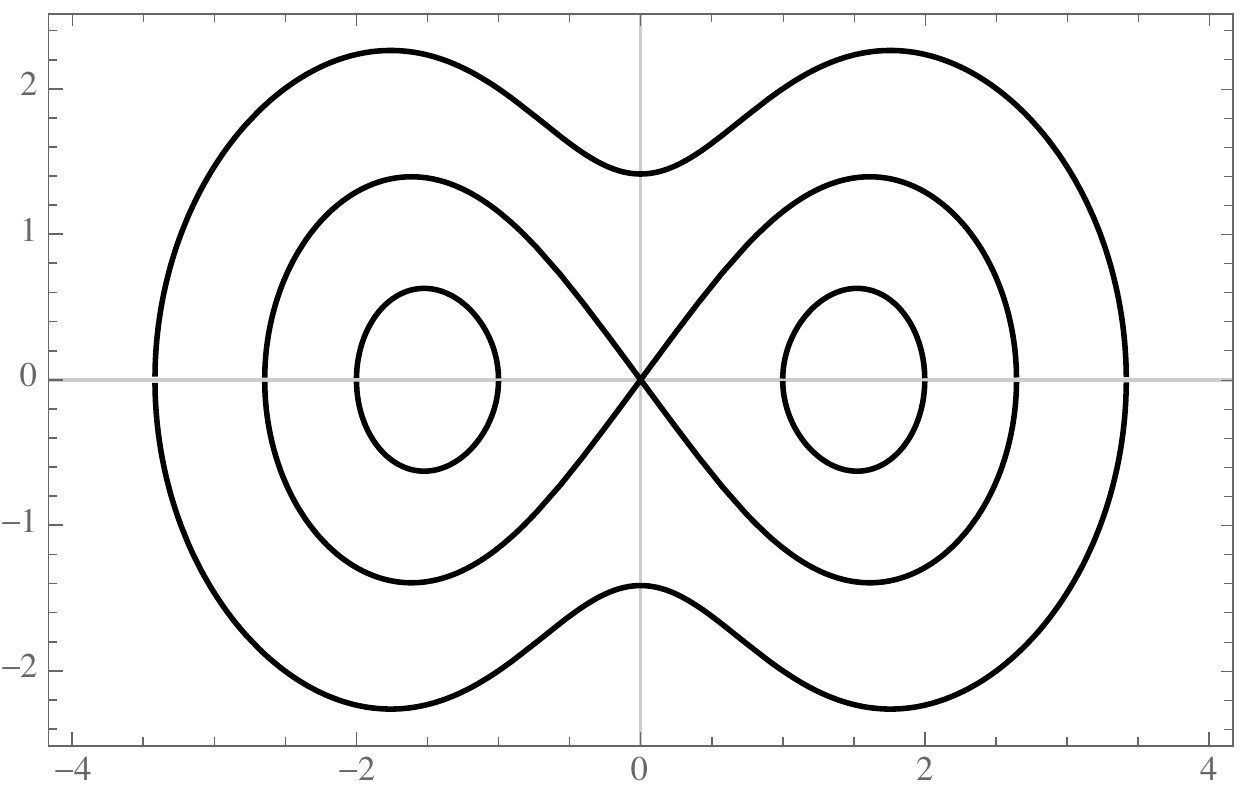}
\put(33.5,33.2){\circle*{1.5}}
\put(69.1,33.2){\circle*{1.5}}
\put(51.3,33.2){\circle*{1.5}}
\linethickness{0.5mm}
\put (50.6,49.9){\vector(10,-1){2}}
\put (69,49.3){\vector(5,1){2}}
\put (34,49.3){\vector(15,-1){2}}
\put (69,40.3){\vector(15,-1){2}}
\put (34,40.3){\vector(15,-1){2}}
\end{overpic}
\end{center}
\end{figure}
%%%%%%%%%%%%%%%%

%%%%%%%%%

\begin{figure}[h]
\caption{Typical level curves of $ H_0$,  in the case  $\kk>0$, $\kk\ge\tfrac18X_3^2$.}
 \label{phasediagram4}
\ \\
\setlength\unitlength{1mm}
\begin{center}

\begin{overpic}[scale=.75]%,grid,tics=5]
{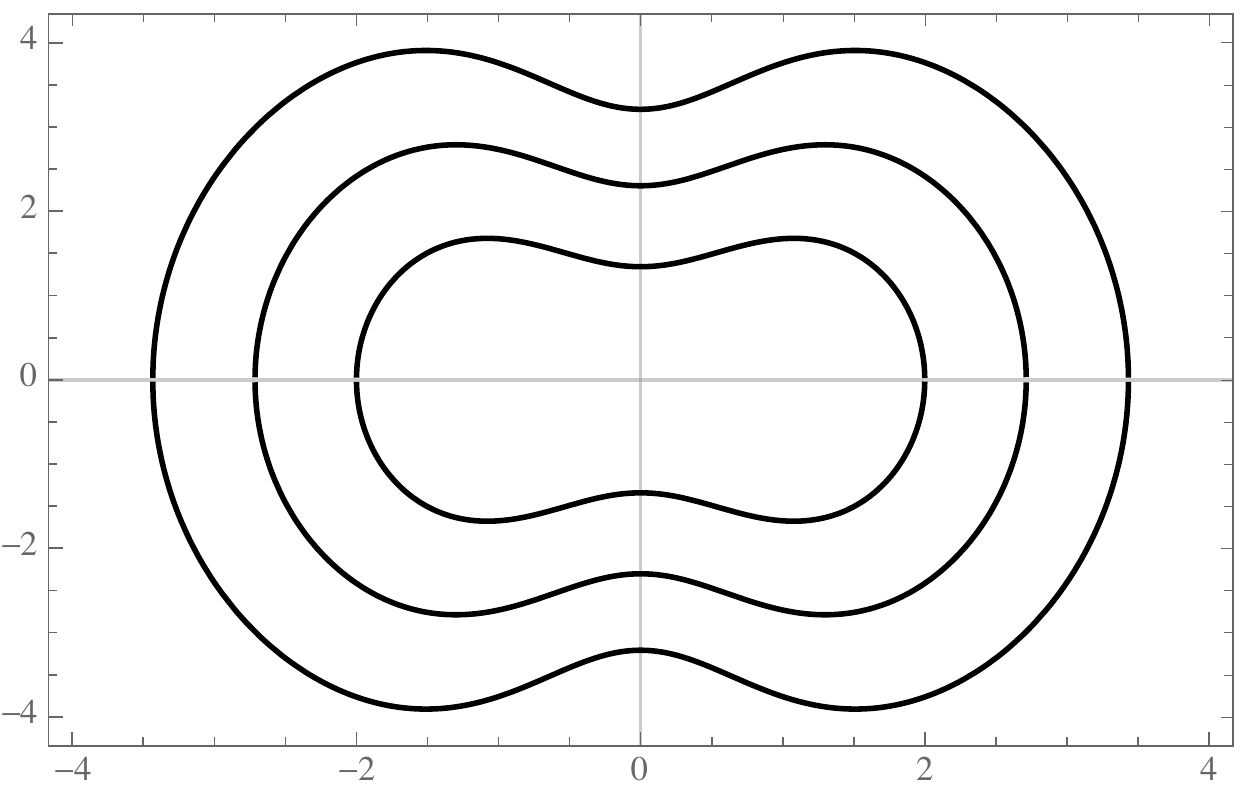}
\linethickness{0.5mm}
\put (51.4,54.8){\vector(1,0){2}}
\put (51.4,48.6){\vector(1,0){2}}
\put (51.4,42.1){\vector(1,0){2}}\put(51.3,33.3){\circle*{1.5}}
\end{overpic}
\end{center}
\end{figure}
%%%%%%%%%%%%%%%%

%%%%%%%%%

\begin{figure}[h]
\caption{Typical level curves of $ H_0$,  in the case  $\kk=0$, $\xi_3\ne0$.}
 \label{phasediagram5}
\ \\
\setlength\unitlength{1mm}
\begin{center}

\begin{overpic}[scale=.75]%,grid,tics=5]
{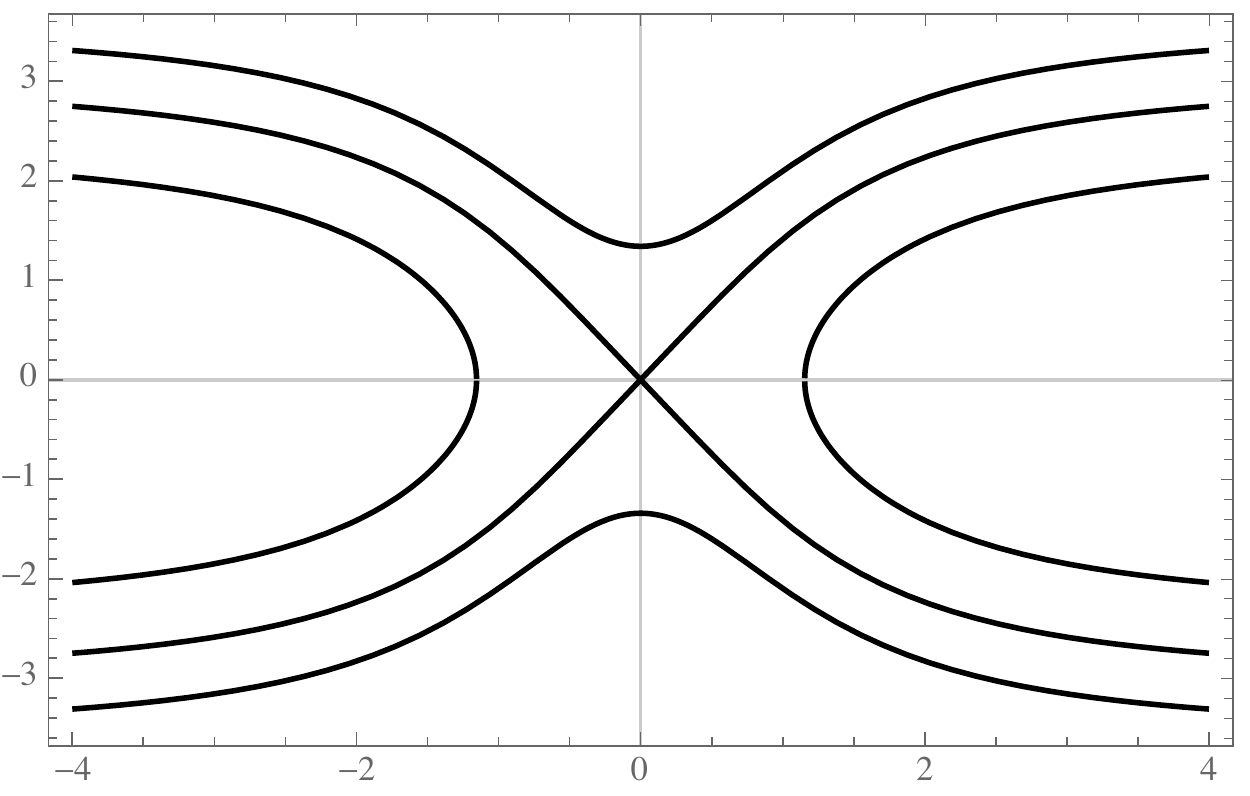}
\linethickness{0.5mm}
\put (51,44){\vector(1,0){2}}
\put (85,53.7){\vector(6,1){2}}
\put (85,47.7){\vector(6,1){2}}
\put (85,12.7){\vector(-6,1){2}}
\put (18,53.7){\vector(6,-1){2}}
\put (18,47.7){\vector(6,-1){2}}
\put (18,12.7){\vector(-6,-1){2}}
\put (51,22.5){\vector(-1,0){2}}
\put(51.3,33.3){\circle*{1.5}}
\end{overpic}
\end{center}
\end{figure}
%%%%%%%%%%%%%%%%

%%%%%%%%%

\begin{figure}[h]
\caption{Typical level curves of $ H_0$,  in the case  $\kk=0$, $\xi_3=0$.
The $\xi_1$-axis is a line of equilibrium points.}
 \label{phasediagram6}
\ \\
\setlength\unitlength{1mm}
\begin{center}

\begin{overpic}[scale=.75]%,grid,tics=5]
{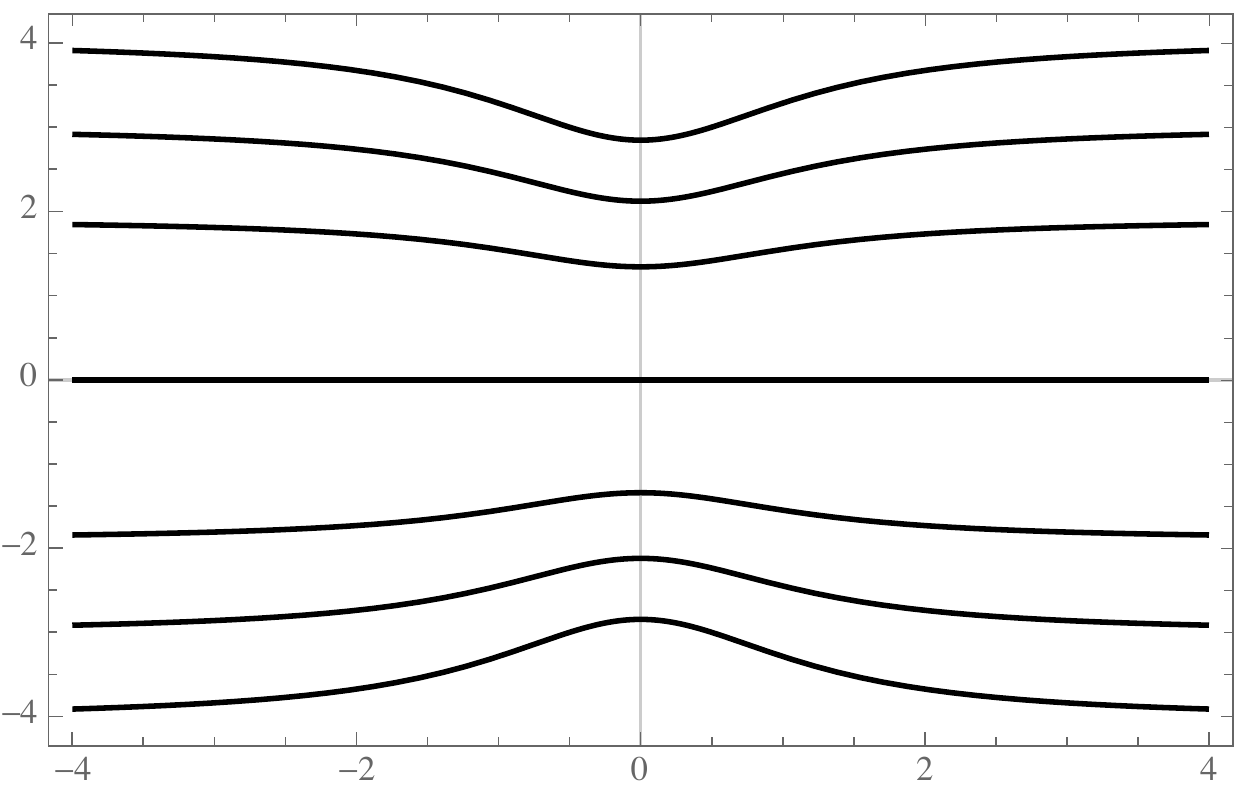}
\linethickness{0.5mm}
\put (51.2,52.4){\vector(1,0){2}}
\put (51.2,47.4){\vector(1,0){2}}
\put (51.2,42.2){\vector(1,0){2}}
\put (51,24.2){\vector(-1,0){2}}
\put (51,19){\vector(-1,0){2}}
\put (51,14){\vector(-1,0){2}}
\end{overpic}
\end{center}
\end{figure}
%%%%%%%%%%%%%%%%

\section{The Lagrange multiplier}

We now obtain an alternate expression for $\Lambda(A,B)$.

\begin{lemma}
\label{nonid}
If $(A,B)\in\dd$, then
\[
\Lambda(A,B)=\frac{4X_1(A,B)+2X_2(A,B)^2-2X_3(A,B)^2}{|A|^4}.
\]
\end{lemma}
\begin{proof}
If $(A,B)\in\dd$, then $\det A=\half\ip{\cof A}{A}=1$ and $\ip{\cof A}{B}=0$,
and so Lemma \ref{metaid} gives
\[
\ip{ZA}{B}\ip{AZ}{B}
=|B|^2+\det{B}\;|A|^2.
\]
By Definition \ref{exdef}, this is equivalent to
\[
-X_2(A,B)^2+X_3(A,B)^2=2X_1(A,B)-\kk|A|^2+\det{B}\;|A|^2.
\]
The lemma now follows from Definition \ref{lamdef}.
\end{proof}

\begin{theorem}
\label{litlam2}
Let $\kk\ge0$.
If $A\in C^0(\rr,\sltwo)\cap C^2(\rr,\mm^2)$ satisfies  \eqref{mainode}
for some $\lambda\in C^0(\rr,\rr)$, then
\[
\lambda(t)=\Lambda(A(t),\dot A(t))=\frac{4X_1+2X_2^2-2X_3^2}{|A(t)|^4},
\]
with
\[
X_i=X_i(A(0),\dot A(0)),\quad i=1,2,3.
\]
\end{theorem}

\begin{proof}
This follows by Lemma \ref{litlam}, Lemma \ref{nonid}, and Theorem \ref{inv1}.
\end{proof}

\begin{corollary}
\label{pressuresign}
The sign of the pressure is preserved under affine motion.
\end{corollary}

\begin{proof}
This is a consequence of Lemma \ref{affinehommhd} (see Remark \ref{pressuresignremark}) and Theorem \ref{litlam2}.
\end{proof}

We now consider the case of vanishing pressure.

\begin{lemma}
\label{vanpress}
Let $A\in C^0(\rr,\sltwo)\cap C^2(\rr,\mm^2)$ be the solution of \eqref{mainode} with initial data
 $(A_0,B_0)\in\dd$ such that  
 \[
 2X_1(A_0,B_0)+X_2(A_0,B_0)^2-X_3(A_0,B_0)^2=0.
 \]  
If $\kk=0$, then
\[
A(t)=B_0t+A_0,
\]
and if $\kk>0$, then 
\[
A(t)=(\cos\sqrt\kk t )A_0+\tfrac1{\sqrt\kk}(\sin\sqrt\kk t )B_0.
\]

\end{lemma}

\begin{proof}
By Theorem \ref{litlam2} and the assumption on the invariants, we have $\lambda(t)=0$.
Thus, \eqref{mainode}, reduces to the linear equation $\ddot A+\kk A=0$ whose solutions are as stated above.
\end{proof}

\begin{remark}
When $\kk>0$, the solution $A(t)$, and hence the motion $x(t,y)$,
 is $2T$-periodic with $T=\pi/\sqrt\kk$, while $|A(t)|$ and  the fluid domain $\omt$
is $T$-periodic.  

\end{remark}

\begin{remark}
When $\kk=0$, we obtain geodesics in $\sltwo$ and in $\mm^2$.
\end{remark}

\section{Rigid solutions}

\begin{definition}
A solution $A\in C^0(\rr,\sltwo)\cap C^2(\rr,\mm^2)$ of \eqref{mainode} is called rigid if
$|A(t)|$ is constant for all $t\in\rr$.  A solution $A$ is called a rigid rotation if $A(t)\in\sotwo$,
for all $t\in\rr$.
\end{definition}

Recall that by Lemma \ref{so2.1}, $A(t)\in\sotwo$ if and only if $|A(t)|^2=2$, so a rigid rotation is a rigid solution,
according to the definitions.  Any equilibrium solution is rigid.

As noted in Theorems \ref{mainthm1} and \ref{mainthm2},  the fluid domains $\omt$ of a solution $A(t)$
form a  family of ellipses whose
dimensions depend only on $|A(t)|$.
So for rigid motion, the  $\omt$ are constant up to rotation, and $\omt=\bb$ for a rigid rotation.

\begin{theorem}
\label{rigidsolthm}
Fix $\kk\ge0$.
Let $A\in C^0(\rr,\sltwo)\cap C^2(\rr,\mm^2)$ be a solution of \eqref{mainode} with initial data
$(A_0,B_0)\in\dd$, and set $X_i=X_i(A_0,B_0)$, $i=1,2,3$.

\begin{enumerate}[i.]
\item
\label{equil}
$A$ is an equilibrium if and only if $X_1=\kk$.
\item
\label{rigidrot}
$A$ is a rigid rotation if and only if $A_0\in\sotwo$ and $X_1=\kk+\tfrac14X_3^2$.  % and $X_2=0$.
In this case, $A(t)=U(\half X_3t)A_0$.
\item
\label{allrigid}
$A$ is  rigid with $|A_0|^2>2$ if and only if
\begin{equation}
\label{rigcond1}
X_1=\half\left(\frac{X_2^2}{\half|A_0|^2-1}+\frac{X_3^2}{\half|A_0|^2+1}\right)+\khalf|A_0|^2
\end{equation}
and
\begin{equation}
\label{rigcond2}
\frac{X_2^2}{(\half|A_0|^2-1)^2}+\frac{X_3^2}{(\half|A_0|^2+1)^2}=2\kk.
\end{equation}
\end{enumerate}
\end{theorem}

\begin{proof}
If $A$ is an equilibrium, then $\dot A=0$, and so in particular, $B_0=0$.  This implies that
\begin{equation}
\label{invvals}
X_1=\khalf|A_0|^2\isp{and}X_i=0,\quad  i=2,3.
\end{equation}
If $\kk=0$, then we have verified that $X_1=\kk$.  So we may  assume that $\kk>0$.
By Lemma \ref{nonid} and \eqref{invvals}, we see that
\[
\Lambda(A_0,B_0)=4X_1/|A_0|^4=\kk^2/X_1.
\]
Now $A$ solves \eqref{odes}, so setting $\ddot A=0$ and then $t=0$, we get  
\[
\kk A_0= \Lambda(A_0,B_0)\cof A_0=(\kk^2/X_1)\cof A_0.
\]
Since the cofactor map is norm-preserving, we find
\[
\kk=\kk^2/X_1.
\]
The assumption $\kk>0$ implies that $X_1=\kk$.

Conversely, if $X_1=\kk$, then 
\[
\half |\dot A|^2= X_1(A,\dot A)-\khalf |A|^2=X_1-\khalf|A|^2=\khalf(2-|A|^2)\le0,
\]
by Lemma \ref{so2.1}.  This shows that $\dot A=0$, so $A$ is an equilibrium.
This completes the proof of statement \ref{equil}.

Let $A_0\in\sotwo$.  Then $(A_0,B_0)\in\dd_0$.  By Lemma \ref{so2x2}, $X_2=0$.
By Lemma \ref{phigammapsizero}, we may choose
$(q(0),\xi(0))\in\rr^3\times\rr^3$ such that
\[
(A_0,B_0)=\Phi\circ\Gamma\circ\Psi_0(q(0),\xi(0)),\quad X_1=H_0(q(0),\xi(0)),\quad \xi_2(0)=0,\quad X_3=\xi_3(0).
\]
Since $A_0\in\sotwo$, we also have $q_1(0)^2=\half|A_0|^2-1=0$.  Thus,
\begin{equation}
\label{coordid}
q(0)=(0,q_2(0),q_3(0))\isp{and} \xi(0)=(\xi_1(0),0,X_3).
\end{equation}
Let $(q(t),\xi(t))$ be the solution of \eqref{Hamsys3} with initial data $(q(0),\xi(0))$.
Then by Theorem \ref{mainthm2},
\begin{equation}
\label{aformsotwo}
A(t)=\tfrac1{\sqrt2}U \left(\half(q_3(t)+q_2(0))\right) \left(\rho(t) I+q_1(t) K\right)U \left(\half(q_3(t)-q_2(0))\right).
\end{equation}
We also note that since $q_1(0)=0$
\begin{equation}
\label{x1id}
X_1=H_0(q(0),\xi(0))=\tfrac14(\xi_1(0)^2+X_3^2)+\kk.
\end{equation}

Given that \eqref{coordid} and \eqref{x1id} hold, we can now assert  that the following statements are equivalent:
\begin{itemize}
\item
$X_1=\kk+\tfrac14X_3^2$,
\item
$\xi_1(0)=0$,
\item
$q(t)=(0,q_2(0),\half X_3t+q_3(0))$ and $\xi(t)=(0,0,X_3)$,
\item
$q_1(t)=0$ for all $t\in\rr$,
\item
$|A(t)|^2=2$ for all $t\in\rr$,
\item
$A(t)\in\sotwo$ for all $t\in\rr$.
\end{itemize}

Finally, we note that if $A(t)\in\sotwo$, for all $t\in\rr$, then $q_1(t)=0$, $\rho(t)=\sqrt2$,  and
the formula \eqref{aformsotwo}
reduces to $A(t)=U(q_3(t))=U(\half X_3t)A_0$.
This proves statement \ref{rigidrot}.

Next, suppose that $|A_0|^2>2$. Then $(A_0,B_0)\in\dd\setminus\dd_0$.
 Let us also assume for the moment that $X_2\ne0$.
By Lemma \ref{phigammapsi}, there exists $(q(0),\xi(0))\in\rr_+^3\times\rr^3$ such that
\[
(A_0,B_0)=\Phi\circ\Gamma\circ\Psi(q(0),\xi(0)),
\]
and by and Lemma \ref{qxiinv1}
\begin{equation}
\label{qxiinv2}
 X_1=\widetilde H(q(0),\xi(0)),\quad 
X_2=\xi_2(0),\quad X_3=\xi_3(0).
\end{equation}
Let $(q,\xi)\in C^1(\rr,\rr_+^3\times\rr^3)$ be the solution of \eqref{Hamsys2}
with initial data $(q(0),\xi(0))$.
 By Theorem \ref{mainthm1}, we have
 \[
 A(t)=\varphi\circ\psi\circ q(t).
 \]
 
 Since $q_1(0)^2=\half|A_0|^2-1$, we see from  \eqref{qxiinv2} and Definition \ref{Ham2}  that \eqref{rigcond1}
 is equivalent to $\xi_1(0)=0$.
 By \eqref{Hamsys2}, we also note that \eqref{rigcond1} and \eqref{rigcond2} are equivalent to $\dot \xi_1(0)=0$.
 Therefore, if \eqref{rigcond1}, \eqref{rigcond2} hold, then the unique solution of \eqref{Hamsys2}
 is
 \[
 q(t)=\left(q_1(0),\frac{X_2}{q_1(0)^2}t+q_2(0),\frac{X_3}{2+q_1(0)^2}t+q_3(0)\right),\quad
 \xi(t)=\left(0,X_2,X_3\right).
 \]
Since $q_1(t)$ is constant, $A$ is rigid.

Conversely, if $A$ is rigid with $|A(t)|^2=|A_0|^2>2$, then $q_1(t)=q_1(0)>0$ and $\dot q_1(t)=0$,
for all $t\in\rr$.  It follows from \eqref{Hamsys2} that 
$\xi_1(t)=0$, for all $t\in\rr$.  As noted above, this implies that \eqref{rigcond1} holds.  We also see that
$\dot \xi_1(t)=0$, for all $t\in\rr$, and  
$q_1(t)=q_1(0)>0$, \eqref{Hamsys2}  implies that  \eqref{rigcond2} holds.

If $X_2=0$, then the result follows from the analogous argument using 
Lemma \ref{phigammapsizero} and Theorem \ref{mainthm2}.
\end{proof}

\section{Asymptotic behavior for MHD, $\kk>0$}
\label{mhdsec}
\begin{theorem}
\label{kpoptions}
Fix $\kk>0$.  Let $A\in\curve$ be a solution of \eqref{mainode}.
Then either
\begin{itemize}
\item
$A$  is rigid,  
\item
$|A|$ is nonconstant and $T$-periodic with $T>0$,   
\item
or $|A|>2$ and
$|A(t)|\searrow2$, as $t\to\pm\infty$.
\end{itemize}
 \end{theorem}
 
 \begin{proof}
Since $\half |A|^2=1+q_1^2$ by Theorems \ref{mainthm1} and  \ref{mainthm2},
this summarizes Theorems \ref{htlevelsets}  and \ref{hzlevelsets}, when $\kk>0$.
\end{proof}
 
 The next result provides a detailed analysis of the third option in Theorem \ref{kpoptions}.
 
\begin{theorem}
\label{kpinvman}
Fix $\kk>0$.  If $X_1=\kk+\tfrac14X_3^2$ with $X_3^2>8\kk$, then the set
\begin{equation}
\label{invman1}
\ww=\{(A,B)\in\dd: X_i(A,B)=X_i,\; i=1,3,\; X_2(A,B)=0\}
\end{equation}
is nonempty and invariant under the flow 
of \eqref{mainode} in  $ C^0(\rr,\sltwo)\cap C^2(\rr,\mm^2)$.

Let $0<\mu<\half(X_3^2-8\kk)^{1/2}$.
If $A\in\curve$ is a solution of \eqref{mainode}  with initial data $(A_0,B_0)\in\ww$, then
there exist a time $T>0$ and phases $\theta_\pm$ such that 
\[
\left| \left(\frac{d}{dt}\right)^j \left[
A(t)-U\left(\tfrac12X_3t+\theta_{+}\right)\right]\right|\lesssim e^{-\mu t},\quad j=0,1,\quad 
\ip{A(t)}{\dot A(t)}<0,\quad t>T,
\]
and
\[
\left| \left(\frac{d}{dt}\right)^j\left[
A(t)-U\left(\tfrac12X_3t+\theta_{-}\right)\right]\right|\lesssim e^{\mu t},\quad j=0,1,\quad 
\ip{A(t)}{\dot A(t)}>0,\quad t<-T.
\]

\end{theorem}

\begin{proof}
As outlined in Theorem \ref{hzlevelsets}, 
the assumptions on the invariants imply that
the set 
\[
\cc(X_3)=\{(q,\xi)\in\rr^3\times\rr^3: \xi_2=0,\; \xi_3=X_3,\;H_0(q_1,\xi_1,X_3)=H_0(0,0,X_3)\}
\]
is the union of  an equilibrium point at $(0,0)$ and two homoclinic orbits.
The image of $\cc(X_3)$ under $\Phi\circ\Gamma\circ\Psi_0$ is equal to $\ww$, and so $\ww\ne\emptyset$.

The invariance of $\ww$ follows by virtue of the invariance of the quantities $X_i(A,\dot A)$, according to Theorem
\ref{inv1}.

Fix $\kk>0$, $X_3>8\kk$ and let $X_1=\kk+\tfrac14X_3^2$.  Suppose that $(q,\xi)\in C^1(\rr,\rr^3\times\rr^3)$ is
a solution of \eqref{Hamsys3} such that $(q_1(t),\xi_1(t))$ parameterizes the homoclinic
orbit in $\cc(X_3)$ with $q_1(t)>0$.  From \eqref{Hamsys3}, we see that this orbit satisfies
\[
q_1(t)\searrow0,\quad \xi_1(t)\nearrow0,\isp{as} t\to\infty,
\]
and
\[
q_1(t)\searrow0,\quad \xi_1(t)\searrow0,\isp{as} t\to-\infty.
\]
Let $\alpha^2=\tfrac14X_3^2-2\kk$.  Then for any $\eps>0$, we can find a $T\gg1$ such that
\begin{equation}
\label{pos1}
0< (|q_1(t)|+| \xi_1(t)|)^3<\eps(\alpha q_1(t)-\xi_1(t)),\quad t>T,
\end{equation}
and
\begin{equation}
\label{pos2}
0< (|q_1(t)|+| \xi_1(t)|)^3<\eps(\alpha q_1(t)+\xi_1(t)),\quad t<-T.
\end{equation}

From \eqref{Hamsys3}, we see that
\[
\dot q_1-\xi_1=\oo(|q_1|^3+|\xi_1|^3),\quad \dot\xi_1-\alpha^2q_1=\oo(|q_1|^3+|\xi_1|^3).
\]
Taking linear combinations of these two equations yields
\[
\frac{d}{dt}(\alpha q_1- \xi_1)+\alpha(\alpha q_1- \xi_1)=\oo(|q_1|+|\xi_1|)^3
\]
and
\[
\frac{d}{dt}(\alpha q_1+ \xi_1)-\alpha(\alpha q_1+\xi_1) =\oo(|q_1|+|\xi_1|)^3.
\]
Choose $0<\mu<\alpha$.  Using \eqref{pos1} and \eqref{pos2}, we can find a $T>0$ such that
\[
\frac{d}{dt}(\alpha q_1- \xi_1)+\alpha(\alpha q_1- \xi_1)<(\alpha-\mu)(\alpha q_1- \xi_1),\quad t>T
\]
and
\[
\frac{d}{dt}(\alpha q_1+ \xi_1)-\alpha(\alpha q_1+\xi_1) >-(\alpha-\mu)(\alpha q_1+ \xi_1),\quad t<-T.
\]
From here, we obtain the bounds
\begin{equation}
\label{ub1}
0<\alpha q_1(t)-\xi_1(t)\lesssim e^{-\mu t},\quad t>T
\end{equation}
and
\begin{equation}
\label{ub2}
0<\alpha q_1(t)+\xi_1(t)\lesssim e^{\mu t},\quad t<-T.
\end{equation}

Going back to \eqref{Hamsys3}, we can now write
\[
q_3(t)=\half X_3t+\theta_++\half X_3\int_t^\infty\frac{q_1(s)^2}{2+q_1(s)^2}ds,\quad 
\theta_+=q_1(0)-\half X_3\int_0^\infty\frac{q_1(s)^2}{2+q_1(s)^2}ds
\]
and
\[
q_3(t)=\half X_3t+\theta_--\half X_3\int^t_{-\infty}\frac{q_1(s)^2}{2+q_1(s)^2}ds,\quad 
\theta_-=q_1(0)+\half X_3\int_{-\infty}^0\frac{q_1(s)^2}{2+q_1(s)^2}ds.
\]
The integrals converge, thanks to \eqref{ub1}, \eqref{ub2}.  Moveover, we have
\begin{equation}
\label{ub3}
|q_3(t)-(\half X_3t+\theta_+)|\lesssim e^{-\mu t},\quad t>T
\end{equation}
and
\begin{equation}
\label{ub4}
|q_3(t)-(\half X_3t+\theta_-)|\lesssim e^{\mu t},\quad t<-T.
\end{equation}
Recycling the estimates  \eqref{ub1}, \eqref{ub2}, \eqref{ub3}, \eqref{ub4}  
in \eqref{Hamsys3} yields identical bounds for the derivatives.

By Theorem \ref{mainthm2}, we can write
\begin{multline*}
A(t)-U(\half X_3t+\theta_+)\\=
U\left(\half(q_3(t)+q_2(0))\right)\;\big((\rho(t)-1)I+q_3(t)K\big)\;U\left(\half(q_3(t)-q_2(0))\right)\\
+\left(U\big(q_3(t)-\half X_3t-\theta_+\big)-I\right))\;
U\big(\half X_3t+\theta_+\big).
\end{multline*}
Thus, applying the bounds \eqref{ub1} and \eqref{ub3}, we obtain
\[
\left|A(t)-U\big(\half X_3t+\theta_+\big)\right|
\le |(\rho(t)-1)I+q_3(t)K|+\left|U\big(q_3(t)-\half X_3t-\theta_+\big)-I\right|\lesssim e^{-\mu t},
\]
for $t>T$.  The corresponding estimate for $t<-T$ is proven in the same way using \eqref{ub2} and
\eqref{ub4}.

We note that
\[
 \ip{A(t)}{\dot A(t)}=\frac{d}{dt}\half|A(t)|^2=\frac{d}{dt}(1+q_1(t)^2)=q_1(t)\dot q_1(t).
 \]
 By \eqref{Hamsys3}, we see that 
 \[
 \sign \ip{A(t)}{\dot A(t)} =\sign \xi_1(t).
 \]

 The argument for the case $q_1(t)<0$ is symmetric.
 
\end{proof}

\begin{remark}
The total phase shift is given by the
expression
\[
\theta_+-\theta_-=-\half X_3\int_{-\infty}^\infty\frac{q_1(s)^2}{2+q_1(s)^2}ds.
\]
\end{remark}

\begin{corollary}
\label{kkpnhim}
Fix $\kk>0$.  If $X_1=\kk+\tfrac14X_3^2$ with $X_3^2>8\kk$, then the set
\begin{equation}
\label{rotinvman}
\rrr(X_3)=\{(A,B)\in\ww: A\in\sotwo\}
\end{equation}
corresponds to the orbit of the rigid rotation
$U\left(\tfrac12X_3t\right)$.  The set $\ww\setminus\rrr(X_3)$ is a stable and unstable manifold for $\rrr(X_3)$.
Every solution orbit  $(A(t),A'(t))$ in $\ww\setminus\rrr(X_3)$ is homoclinic to $\rrr(X_1)$, that is,
\[
\lim_{|t|\to\infty}e^{\mu|t|}\dist [(A(t),A'(t)),\rrr(X_3)]=0,
\]
for some $\mu>0$.
\end{corollary}

\begin{proof}
The function $A(t)=U(\half X_3t)$ is a rigid rotation, by Theorem \ref{rigidsolthm}.
  Its orbit is $\{(A,\half X_3ZA):A\in\sotwo\}$, and it is easily
 verified that this set coincides with $\rrr(X_3)$.
 
 The final statement is a consequence of Theorem \ref{kpinvman}.
 
\end{proof}

\begin{remark}
The corollary shows that the rigid rotational solutions are unstable within the class
of affine motions when $X_3^2>8\kk$.
\end{remark}

\begin{theorem}
\label{multiperiodic}
If $A\in\curve$ is a solution of  \eqref{mainode}
such that $|A|$  is nonconstant and $T$-periodic for some $T>0$, then
the solution has the form
\[
 A(t) =U\left(\half(\omega_1+\omega_2)t\right)\hat A(t) \;U\left(\half(\omega_1-\omega_2)t\right),
 \]
 where $\hat A(t)$ is
is $T$-periodic if $A(t)\notin\sotwo$, for all $t\in\rr$, and $2T$-periodic if $A(t)\in\sotwo$,
for some $t\in\rr$.
The frequencies are defined by
\begin{align*}
&\omega_1=
\frac{1}{T}\int_0^T\frac{X_3}{\half|A(t)|^2+1}\;dt
\intertext{and}
&\omega_2=
\begin{cases}
0,&X_2=0\\
\displaystyle \frac{1}{T}\int_0^T\frac{X_2}{\half|A(t)|^2-1}\;dt, &X_2\ne 0,
\end{cases}
\end{align*}
where $X_i=X_i(A,\dot A)$, $i=2,3$.
\end{theorem}

\begin{proof}
By Theorems \ref{mainthm1} and \ref{mainthm2}, we can write
\begin{equation}
\label{Aform1.1}
A(t)=\tfrac1{\sqrt2}U \left(\half(q_3(t)+q_2(t))\right) \left(\rho(t) I+q_1(t) K\right)U \left(\half(q_3(t)-q_2(t))\right),
\end{equation}
where the $q_i(t)$ are obtained by solving \eqref{Hamsys2}, if $X_2\ne0$, or \eqref{Hamsys3}, if $X_2=0$, with
appropriate initial data.  

We have that $q_1^2=\half|A|^2-1$ is nonconstant and $T$-periodic with $T>0$.
Note $q_1(t)=0$ if and only if $A(t)\in\sotwo$.
By Theorems \eqref{htlevelsets} and \ref{hzlevelsets}, we see that either $q_1(t)\ne0$, for all $t\in\rr$,
whence $q_1$ is also $T$-periodic, or $q_1(t)=0$ for a sequence of times $t_0+kT$, $k\in\zz$, whence
by symmetry, $q_1$ is $2T$-periodic.
Thus, we have that 
\[
\rho(t) I+q_1(t) K
\]
is either $T$- or $2T$-periodic, depending on whether
$A$ passes through $\sotwo$ or not.  

Notice that for  each $i=2,3$, the quantity
 $\omega_i$ defined above is the mean of $\dot q_i$ over one period.
 Thus, we have that
\[
q_i(t)-\omega_it
\]
is $T$-periodic.
So the now result follows from \eqref{Aform1.1} by writing
\[
U \left(\half(q_3(t)\pm q_2(t))\right)=U\left(\half(\omega_1\pm\omega_2)t\right) 
U \left(\half(q_3(t)-\omega_1t\pm q_2(t)\mp\omega_2t )\right)
\]
\end{proof}

\begin{remark}
Note that the result holds for rigid solutions.  In this case, the quantity $|A(t)|$ is constant
and thus $T$-periodic for all $T\ge0$.  Any value of $T>0$ can be used in computing the frequencies.
Thus, solutions are generically quasiperiodic.
\end{remark}

\newcommand{\fp}[1]{\{#1\}}

\begin{theorem}
\label{recurrence}
Let $A\in\curve$ be a solution of \eqref{mainode} 
such that the quantity $|A(t)|$ is nonconstant and $T$-periodic for some $T>0$.

For every $N\in\mathbb N$, there exists $\ell\in\{1,\ldots,N^2\}$ depending on $N$ such that
\[
|A(2\ell  T+t)-A(t)|\le 8\pi|A(t)|/N,\isp{for all} t\in\rr.
\]

If $A(t)$ is rigid, then either $A(t)$ is periodic or
the range of $A(t)$ is dense in the sphere of radius $|A(0)|$ in $\sltwo$.  
\end{theorem}

\begin{proof}
By Theorem \ref{multiperiodic}, we may write
\[
A(t)=U(\omega_1t)\hat A(t)U(\omega_2t),
\]
in which $\hat A(t)$ is $2T$-periodic. (If $A(t)$ does not pass through $\sotwo$, then
we know that $\hat A(t)$ is $T$-periodic.)

For every $x\in\rr$, there is a unique $k\in\zz$
such that
\[
\fp{x}\equiv x-2\pi k\in[0,2\pi).
\]
Consider the set of $N^2+1$ ordered pairs
\[
\left\{\;\left(\fp{\omega_12jT},\fp{\omega_22jT}\right):j=0,1,\ldots,N^2\;\right\}
\]
contained in the square $[0,2\pi)\times[0,2\pi)$.  Partition this square into $N^2$ 
congruent subsquares of side $2\pi/N$.  By the pigeonhole 
principle, two of these ordered pairs belong to the same subsquare.  It follows that there exist
$k, \ell \in\zz$ such that
$0\le k<k+\ell \le N^2$ and
\[
|\fp{\omega_i2kT}-\fp{\omega_i2(k+\ell  )T}|\le 2\pi/N,\quad i=1,2.
\]
Thus, there exist $m_i\in\zz$ such that
\[
|\omega_i2\ell  T+2\pi m_i|\le2\pi /N,\quad i=1,2.
\]
Define
\[
\tau_i=\omega_i2\ell  T+2\pi m_i.
\]

For $i=1, 2$ and $t\in\rr$, we have using  Definition \ref{rotmat} and the mean value theorem
\begin{align*}
|U(\omega_i2\ell  T+t)-U(t)|
&=
|U(\omega_i2\ell  T)-I|\\
&=|U(\tau_i)-I|\\
&=\sqrt2[(\cos \tau_i-1)^2+\sin^2\tau_i]^{1/2}\\
&=2(1-\cos\tau_i)^{1/2}\\
&\le2|\tau_i|\\
&\le 4\pi/N.
\end{align*}

For any $t\in\rr$, we have
\begin{align*}
A(2\ell  T+t)&=U(\omega_1(2\ell  T+t))\hat A(2\ell  T+t) U(\omega_2(2\ell  T+t))\\
&=U(\omega_12\ell  T)U(\omega_1t)\hat A(t)U(\omega_2 t)U(\omega_22\ell  T)\\
&=U(\tau_1)A(t)U(\tau_2).
\end{align*}
We now estimate as follows
\begin{align*}
|A(2\ell  T+t)-A(t)|
&=|U(\tau_1)A(t)U(\tau_2)-A(t)|\\
&=|[U(\tau_1)-I]A(t)U(\tau_2)+A(t)[U(\tau_2)-I]|\\
&\le |U(\tau_1)-I||A(t)||U(\tau_2)|+|A(t)||U(\tau_2)-I|\\
&\le 2(4\pi/N)|A(t)|.
\end{align*}
This proves the first statement.

If $A(t)$ is rigid, then $|A(t)|=|A(0)|$, and so by Theorem \ref{multiperiodic}
\[
A(t)=U(\omega_1t)A(0)U(\omega_2t)=U(\fp{\omega_1t})A(0)U(\fp{\omega_2t}).
\]
If  $\omega_1$ and $\omega_2$ are rationally dependent, then $A(t)$
is periodic.  
The curve 
\[
t\mapsto\left(\fp{\omega_1t},\fp{\omega_2t}\right)
\]
represents linear flow on the torus.
If $\omega_1$ and $\omega_2$ are rationally independent, then it is well-known that
the image of the curve is dense in the square
$[0,2\pi)\times[0,2\pi)$.  The set
\[
\{UA(0)V: U, V\in\sotwo\}
\]
coincides with the sphere of radius $|A(0)|$ in $\sltwo$.
Thus, the range of $A(t)$ is dense in this sphere.
\end{proof}

\begin{remark}
The only solutions $A(t)$ for which $|A(t)|$ is not periodic are those 
which are homoclinic to a rigid rotation.
Thus, the result shows that, generically, solutions are recurrent.
\end{remark}

\begin{remark}
Since
\[
|A(t)|\le \left[\tfrac2\kk X_1(A(t),\dot A(t)\right]^{1/2}
\]
and the energy is conserved,
 Theorem \ref{recurrence} shows that 
\[
|A(2\ell  T+t)-A(t)|\lesssim 1/N,\isp{for all} t\in\rr.
\]
\end{remark}

\section{Asymptotic behavior for perfect fluids, $\kk=0$}
\label{pfsec}

\begin{theorem}
Let $\kk=0$, and suppose that $A\in\curve$ is a solution of \eqref{mainode}.
Then either
\begin{itemize}
\item
$A$ is rigid,
\item
$|A(t)|\nearrow\infty$, as $|t|\to\infty$,
\item
$\ip{A(t)}{\dot A(t)}>0$, $t\in\rr$,
$|A(t)|\nearrow\infty$, as $t\to\infty$, and $|A(t)|\searrow2$, as $t\to-\infty$, or
\item
$\ip{A(t)}{\dot A(t)}<0$, $t\in\rr$, $|A(t)|\nearrow\infty$, as $t\to-\infty$, and $|A(t)|\searrow2$, as $t\to\infty$.
\end{itemize}
\end{theorem}

 \begin{proof}
By Theorems \ref{mainthm1}, \ref{mainthm2} and Lemma \ref{Aform1}, we have 
$\half|A|^2=1+q_1^2$, where $q_1$ is obtained by solving \eqref{Hamsys2} or \eqref{Hamsys3}.
Thus, we find that
\begin{equation*}
\label{normderiv}
\ip{A}{\dot A}= \frac{d}{dt}\half|A|^2=2q_1\dot q_1,
\end{equation*}
and so
\begin{equation}
\label{signnormderiv}
\sign\ip{A}{\dot A}=\sign q_1\xi_1.
\end{equation}
Thus, the result summarizes Theorems  \ref{htlevelsets} and \ref{hzlevelsets} when $\kk=0$.
\end{proof}

In analogy with Theorem \ref{kpinvman}, we have

\begin{theorem}
\label{kzinvman}
Fix $\kk=0$.  If $X_1=\tfrac14X_3^2>0$, then the sets
\begin{equation}
\label{invman2}
\wws=\{(A,B)\in\dd: X_i(A,B)=X_i,\; i=1,3,\; X_2(A,B)=0, \ip{A}{B}<0\}
\end{equation}
and
\begin{equation}
\label{invman3}
\wwu=\{(A,B)\in\dd: X_i(A,B)=X_i,\; i=1,3,\; X_2(A,B)=0, \ip{A}{B}>0\}
\end{equation}
are nonempty and invariant under the flow of 
of \eqref{mainode} in  $C^0(\rr,\sltwo)\cap C^2(\rr,\mm^2)$.

Let $0<\mu<\half|X_3|$.
If $A\in\curve$ is a solution of \eqref{mainode}  with initial data $(A_0,B_0)\in\wws$, then
there exist a time $T>0$ and a phase $\theta_+$ such that 
\[
\left|\left(\frac{d}{dt}\right)^j\left[
A(t)-U\left(\tfrac12X_3t+\theta_{+}\right)\right]\right|<e^{-\mu t},\quad j=0,1,\quad
 t>T.
\]
If $(A_0,B_0)\in\wwu$, then
there exist a time $T>0$ and a phase $\theta_-$ such that 
\[
\left|\left(\frac{d}{dt}\right)^j\left[
A(t)-U\left(\tfrac12X_3t+\theta_{-}\right)\right]\right|<e^{\mu t},\quad j=0,1,\quad 
 t<-T.
\]

\end{theorem}

\begin{proof}
As outlined in Theorem \ref{hzlevelsets}, 
the assumptions on the invariants imply that
the sets
\[
\cc_s(X_3)=\{(q,\xi)\in\rr^3\times\rr^3:q_1\xi_1<0,\; \xi_2=0,\; \xi_3=X_3,\;H_0(q_1,\xi_1,X_3)=H_0(0,0,X_3)\}
\]
and
\[
\cc_u(X_3)=\{(q,\xi)\in\rr^3\times\rr^3:q_1\xi_1>0,\; \xi_2=0,\; \xi_3=X_3,\;H_0(q_1,\xi_1,X_3)=H_0(0,0,X_3)\}
\]
are each the union of  two semi-bounded orbits of the system for $(q_1,\xi_1)$.
The images of $\cc_s(X_3)$ and $\cc_u(X_3)$ under $\Phi\circ\Gamma\circ\Psi_0$ are equal to $\wws$ and
$\wwu$, respectively.

The invariance of $\wws$ and $\wwu$ follows by virtue of the invariance of the quantities $X_i(A,\dot A)$, by Theorem
\ref{inv1}, together with \eqref{signnormderiv}.

The remainder of the proof is similar to that of Theorem \ref{kpinvman}, and so we shall omit it.
\end{proof}

\begin{theorem}
\label{k=0pd}
Let $A\in \curve$ be a solution of \eqref{mainode}
with initial data $(A_0,B_0)\in\dd$.
If $|A(t)|^2\nearrow\infty$ as $t\to\infty$, then there exist  $A_\infty,\;B_\infty\in \mm^2$ 
with $B_\infty\ne0$ 
such that for $t>0$,  $j=0,1,2$,
\begin{equation}
\label{errorest}
\left|\left(\frac{d}{dt}\right)^j
\left[A(t)-\left(B_\infty t +A_\infty\right)\right]\right|
\lesssim (1+t)^{-1-j}.
\end{equation}
If $\bar A_\infty,\;\bar B_\infty\in\mm^2$ is any pair such that 
\begin{equation}
\label{asympunique}
\lim_{t\to\infty}|A(t)-(\bar B_\infty t +\bar A_\infty)|=0,
\end{equation}
then $(\bar A_\infty,\bar B_\infty)=(A_\infty,B_\infty)$.

The vectors $A_\infty$, $B_\infty$ satisfy
\[
X_i(A_\infty, B_\infty)=X_i(A_0, B_0)\equiv X_i,\quad i=1,2,3,
\]
and
\begin{equation}
\label{infids}
\ip{ B_\infty}{ \cof A_\infty}=\det B_\infty
=0,\quad \det A_\infty=\frac{X_3^2-X_2^2}{2X_1}.
\end{equation}
If $2X_1+X_2^2-X_3^2=0$, then $(A_\infty,B_\infty)=(A_0,B_0)\in\dd$ and 
\[
A(t)=B_0 t + A_0.
\]

 \end{theorem}

\begin{proof}
By Theorems \ref{mainthm1} and \ref{mainthm2}, we can write $A(t)$ in terms of functions
 $q_i(t)$, $i=1,2,3$,  obtained by solving \eqref{Hamsys2}, if $X_2\ne0$, or \eqref{Hamsys3}, if $X_2=0$, with
appropriate initial data.  Since
\[
q_1^2(t)=\half|A(t)|^2-1\nearrow\infty\isp{as}t\to\infty,
\]
using  Lemmas \ref{qxiinv1} and \ref{phigammapsizero}, we can write 
\[
X_1=\half\left(\frac{(2+q_1(t)^2)\xi_1(t)^2}{2(1+q_1(t)^2)}+\frac{X_2^2}{q_1(t)^2}+\frac{X_3^2}{2+q_1(t)^2}\right),
\]
for $t\gg1$.
Thus, there exists a time $t_0$ such that $\xi_1(t)^2>2X_1$, for all $t>t_0$. 
From \eqref{Hamsys2},  \eqref{Hamsys3}, it follows that
\[
|\dot q_1(t)|\ge |\xi_1(t)|/2\ge (X_1/2)^{1/2},\quad t\ge t_0.
\]
This implies  the lower bound
\[
|A(t)|^2=2(1+q_1(t)^2)\gtrsim (1+t)^2,\quad t\ge 0.
\]

Since $A\in\curve$ solves \eqref{mainode}, we obtain from Lemma \ref{litlam2} that
\[
|\ddot A(t)|\lesssim |A(t)|^{-3}\lesssim (1+t)^{-3},\quad t\ge0.
\]
Thus,  by Lemma 6 of \cite{sideris-2017}, we can write
\begin{align}
\nonumber
&A(t)=B_\infty t + A_\infty + A_1(t),\\
\intertext{with}
\nonumber
& B_\infty= B_0+\int_0^\infty \ddot A(s)ds,\\
\label{asympdefs}
&A_\infty = A_0-\int_0^\infty\int_s^\infty \ddot A(\sigma) d\sigma ds, \\
\nonumber
&A_1(t)=\int_t^\infty\int_s^\infty \ddot A(\sigma)d\sigma ds.
\end{align}
Note that our estimate for $|\ddot A(t)|$ implies that
\[
\left|\left(\frac{d}{dt}\right)^jA_1(t)\right|
\lesssim (1+t)^{-1-j},\quad t\ge0,\quad j=0,\;1\;,2,
\]
thereby proving \eqref{errorest}.

If \eqref{asympunique} holds, then using \eqref{errorest}, we find that
\[
\lim_{t\to\infty}|(B_\infty-\bar B_\infty)t+(A_\infty-\bar A_\infty)|=0,
\]
and uniqueness of the states $(A_\infty, B_\infty)$ follows from this.

Applying \eqref{errorest}, we find 
\[
X_1=\tfrac12|\dot A(t)|^2= \tfrac12|B_\infty+\dot A_1(t)|^2=\tfrac12|B_\infty|^2+O(t^{-1}),\quad t>0.
\]
Sending $t\to\infty$ shows that $X_1=\tfrac12|B_\infty|^2$.

For the other invariants, we have for $i=2,3$,
\begin{multline}
X_i=X_i(A(t),\dot A(t))=X_i(B_\infty t + A_\infty + A_1(t),B_\infty+ \dot A_1(t))\\
=tX_i(B_\infty,B_\infty)+X_i(A_\infty,B_\infty)+O(t^{-1}).
\end{multline}
By Lemmas \ref{symip} and \ref{symasym}, we see that
$X_i(B_\infty,B_\infty)=0$, $i=2,3$, and so letting $t\to\infty$ we obtain
$X_i=X_i(A_\infty,B_\infty)$, $i=2,3$.

Since $A(t)\in\sltwo$, we get from Lemma \ref{detcof}
\begin{align*}
2=&\;2\det A(t)\\
=&\ip{A(t)}{\cof A(t)}\\
=&\;t^2\ip{B_\infty}{\cof B_\infty}+2t\ip{A_\infty}{\cof B_\infty}+O(1)\\
=&\;2t^2\det B_\infty+2t\ip{A_\infty}{\cof B_\infty}+O(1).%\\
\end{align*}
This implies that
\[
\det B_\infty=0\isp{and}\ip{A_\infty}{\cof B_\infty}=0.
\]
So by Lemma \ref{metaid}, we get
\[
-X_2^2+X_3^2=\ip{ZA_\infty}{B_\infty}\ip{A_\infty Z}{B_\infty}
=\half\ip{\cof A_\infty}{A_\infty}|B_\infty|^2=\det A_\infty \; 2X_1.
\]
This verifies the statements \eqref{infids}.

If $2X_1+X_2^2-X_3^2=0$, then $\ddot A(t)=0$, by
Theorem \ref{litlam2}.  By \eqref{asympdefs},  we obtain $A_1(t)=0$ and $(A_\infty,B_\infty)=(A_0,B_0)$,
so that $A(t)=B_\infty t+A_\infty$.
\end{proof}

\begin{remark}
In Theorem \ref{asympunique}, if $2X_1+X_2^2-X_3^2\ne0$, 
then  $A_\infty\notin\sltwo$.  Hence  $(A_\infty,B_\infty)\notin\dd$,
and $B_\infty t+A_\infty\notin C^0(\rr,\sltwo)$.
\end{remark}

\begin{remark}
An analogous result holds when $|A(t)|^2\nearrow\infty$, as $t\to-\infty$.
\end{remark}

\section{The picture in the tangent space}
\label{Tansp}

In this  final section, we examine the following question: For a given initial position
$A\in\sltwo$, which initial velocities $B\in\tasltwo$ launch 
to a solution with vanishing pressure,
a rigid solution, or a solution on an invariant manifold of the rigid rotations?
We shall translate the conditions involving the invariants 
$X_i(A,B)$ and the magnitude $|A|$ given in Theorem \ref{litlam2}, Theorem \ref{rigidsolthm},
and Theorems \ref{kpinvman}, \ref{kzinvman}, respectively, into local coordinates in 
$\tasltwo$.

We first suppose that $A\in\sltwo\setminus\sotwo$
so that $(A,B)\in\dd\setminus\dd_0$.
 Lemma \ref{phigammapsi} says that there exists $(q,\xi)\in\rr_+^3\times\rr^3$
such that
\[
(A,B)=\Phi\circ\Gamma\circ\Psi(q,\xi)=(\varphi\circ\psi(q), T(q)\xi).
\]
Thus, the columns of $T(q)$ span $\tasltwo$, and these vectors are orthogonal
since the metric $h(q)$ in these coordinates is diagonal.  Let us normalize the coordinates
by setting 
\[
\hat \xi=h(q)^{1/2}\xi.
\]
Then
\[
T(q)\xi=T(q)h(q)^{-1/2}\hat \xi\equiv \hat T(q)\hat\xi,
\]
where the columns of $\hat T(q)$ are orthonormal.  Denote this 
orthonormal basis for $\tasltwo$ by $\{\tau_i(A)\}_{i=1}^3$.
In these coordinates, we have according to Lemma \ref{qxiinv1}
\[
X_1=\half|\hat\xi|^2+\kk(1+q_1^2),\quad X_2= q_1\hat\xi_2,\quad X_3=(2+q_1^2)^{1/2}\hat\xi_3,
\]
where $q_1^2=\half|A|^2-1>0$ is fixed.

{\bf Solutions with vanishing pressure.}
Theorem \ref{litlam2} shows that the Lagrange multiplier (and hence also the pressure) vanishes
if and only if
\[
2X_1+X_2^2-X_3^2=0,
\]
which is equivalent to the equation
\[
\frac{\hat\xi_1^2}{1+q_1^2}+\hat\xi_2^2-\hat\xi_3^2+2\kk=0.
\]
This relation describes
a two-sheeted hyperboloid when $\kk>0$ and a cone when $\kk=0$.
 Note that the region of positive pressure is 
 \[
 \frac{\hat\xi_1^2}{1+q_1^2}+\hat\xi_2^2-\hat\xi_3^2+2\kk>0.
 \]

{\bf Rigid solutions.}
By Theorem \ref{rigidsolthm}\ref{allrigid}, the rigid solutions in $\sltwo\setminus\sotwo$ are
characterized by the relations \eqref{rigcond1} and \eqref{rigcond2}
which reduce to
\[
\hat\xi_1=0,\quad \frac{\hat\xi_2^2}{q_1^2}+\frac{\hat\xi_3^2}{2+q_1^2}-2\kk=0.
\]
This represents an ellipse in the $\tau_2,\tau_3$ plane when $\kk>0$.
When $\kk=0$, the condition degenerates, and the only rigid solution
is an equilibrium.

{\bf Stable and unstable manifolds of the rigid rotations.}
When $\kk>0$, the collection of rigid rotations $\bigcup\{\rrr(X_3):X_3^2>8\kk\}$
have an invariant manifold $\bigcup\{\ww:X_3^2>8\kk\}$,
where $\ww$ and $\rrr(X_3)$ were defined in \eqref{invman1} and \eqref{rotinvman}, respectively.
In local coordinates, the defining conditions for this invariant manifold are
\[
\half\hat\xi_1^2+\kk q_1^2=\tfrac14q_1^2\hat \xi_3^2,
\quad \xi_2=0.
\]
If we fix $A\in\sltwo\setminus\sltwo$, then $q_1>0$ and
this defines a hyperbola in the $\tau_1,\tau_3$ plane.
When $\kk=0$, we obtain a pair of lines
\[
|\hat\xi_1|=(q_1/\sqrt2)|\hat\xi_3|,\quad \hat\xi_2=0.
\]
The interested reader can verify that $\sign\ip{A}{B}=\sign \hat\xi_1$, so that  the stable
directions $B\in\tasltwo$ with $(A,B)\in\wws$  correspond to points with $\xi_1<0$ on these lines
while the unstable directions $B\in\tasltwo$ with $(A,B)\in\wwu$  correspond to points with $\xi_1>0$,
(see \eqref{invman2}, \eqref{invman3}).

These sets are illustrated in Figure \ref{tangentspace} when $\kk>0$ and
in Figure \ref{tangentspace0} when $\kk=0$.

%%%%%%%%%

\begin{figure}[ht]
\caption{Distinguished directions in $\tasltwo$ for a fixed $A\in\sltwo\setminus\sotwo$, with  $\kk>0$.
The branch of   pressureless directions in the half space $\xi_3<0$ is not shown.}
 \label{tangentspace}
\ \\
\setlength\unitlength{1mm}
\begin{center}

\begin{overpic}[scale=.75]%,grid,tics=5]
{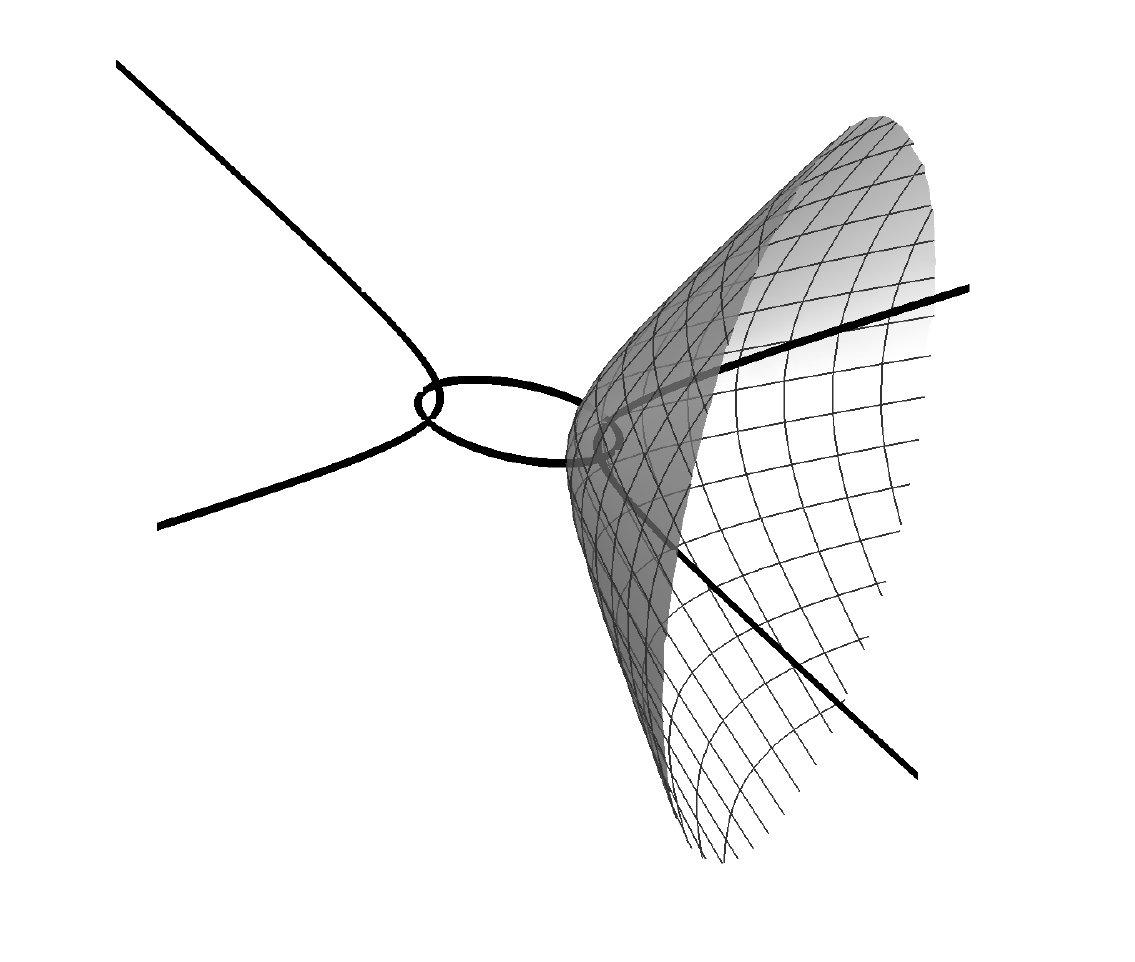}

\put(12,49){\text{Rigid}}
\put(25,50){\vector(1,0){11}}
\put(15,15){\text{Vanishing pressure}}
\put(-15,78){\text{Homoclinic}}
\put(83,15){\text{Homoclinic}}
\put(46,48){\vector(-1,-1){10}}
\put(46,48){\vector(0,1){20}}
\put(46,48){\vector(15,-4){40}}
\put(41,70){$\tau_1(A)$}
\put(30,32){$\tau_2(A)$}
\put(88,36){$\tau_3(A)$}
\end{overpic}
\end{center}
\end{figure}
%%%%%%%%%%%%%%%%

%%%%%%%%%

\begin{figure}[ht]
\caption{Distinguished directions in $\tasltwo$ for a fixed $A\in\sltwo\setminus\sotwo$, with  $\kk=0$.
The cone of   pressureless directions in the half space $\xi_3<0$ is not shown.}
 \label{tangentspace0}
\ \\
\setlength\unitlength{1mm}
\begin{center}

\begin{overpic}[scale=.75]%,grid,tics=5]
{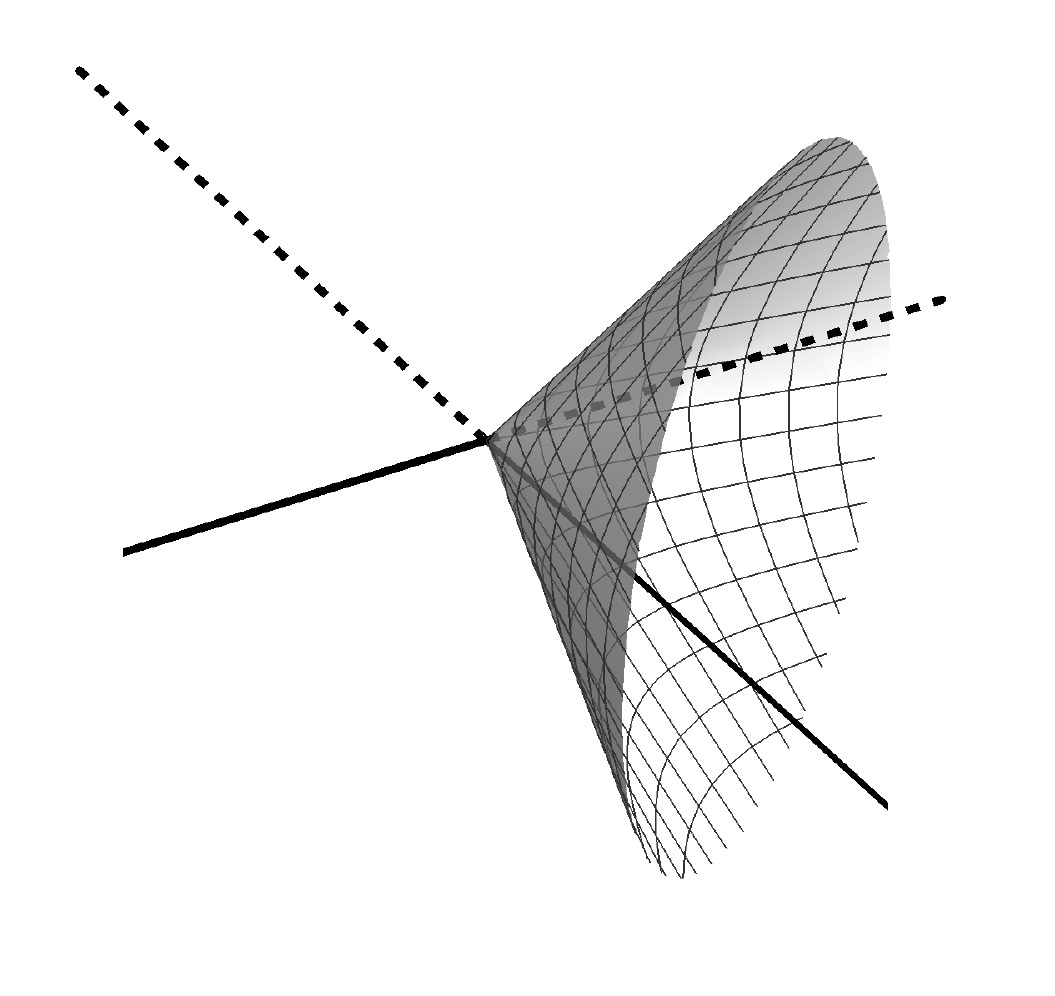}

\put(12,49){\text{Stable}}
\put(13,15){\text{Vanishing pressure}}
\put(-10,78){\text{Unstable}}
\put(46,53){\vector(-1,-1){10}}
\put(46,53){\vector(0,1){20}}
\put(46,53){\vector(15,-4){40}}
\put(41,75){$\tau_1(A)$}
\put(30,37){$\tau_2(A)$}
\put(88,41){$\tau_3(A)$}
\end{overpic}
\end{center}
\end{figure}
%%%%%%%%%%%%%%%%

The situation for $A\in\sotwo$ is fairly simple.  We can use the orthonormal basis
\[
\tau_1(A)=\tfrac1{\sqrt2}AK,\quad \tau_2(A)=\tfrac1{\sqrt2}AM,\quad \tau_3(A)=\tfrac1{\sqrt2}AZ,
\]
in $\tasltwo$.
In this case, if 
\[
B=\hat\xi_i\tau_i(A)\in\tasltwo,
\]
then
\[
X_1=\half |\hat\xi|^2+\kk,\quad X_2=0,\quad X_3=\sqrt2\hat\xi_3.
\]
The solutions with vanishing pressure are 
\[
\hat\xi_1^2+\hat\xi_2^2-\hat\xi_3^2-2\kk=0,
\]
consistent with sending $q_1\to0$ above.  There are no  directions $B\in\tasltwo$
so that $(A,B)$ launches from $\dd_0$ to an invariant manifold
in $\dd\setminus\dd_0$.  Any $B\in\lspan \tau_3(A)$ leads to
a rigid rotation.

\clearpage
\appendix 
\section{Glossary of notation}
\label{glossary}

\ \\

\begin{tabular}{cl p{3in}}
Symbol & Reference &  Description\hfil \\
\ \\
$\mm^2$ &Def \ref{2b2mat} & vector space of $2\times2$ matrices over $\rr$\\
$\ip{\cdot}{\cdot}$ &Def \ref{2b2mat} & Euclidean inner product on $\mm^2$ or on $\rr^n$\\
$I, Z,K,M$ & Def \ref{zdef} & orthogonal basis vectors in $\mm^2$\\
$\sltwo$ & Def \ref{sltsot} & special linear group\\
$\sotwo$ & Def \ref{sltsot} & special orthogonal group\\
$U(\theta)$ & Def \ref{rotmat} & parameterization of $\sotwo$\\
$\cof$ & Lem \ref{cofiso} & cofactor map\\
$\tasltwo$ & Lem \ref{normalvector} & tangent space at $A\in\sltwo$\\
$\dd$ & Def \ref{tanbund} & tangent bundle / phase space \\
 $\dd_0$ & Def \ref{tanbund} & subset of $\dd$\\
$\varphi(x)$ & Def \ref{varphidef}&immersion\\
$g(x)$ & Lem \ref{metric} & metric\\
$\Phi(x,y)$&Def \ref{bigphidef}&tangent bundle map\\
$\sllatwo$ & Def \ref{sllatwodef} & special linear Lie algebra\\
$L(A,B)$ & Def \ref{velgrad} & velocity gradient map\\
$\Lambda(A,B)$ & Def \ref{lamdef} &  Lagrange multiplier\\
$X_i(A,B)$ & Def \ref{exdef} & invariant quantities\\
$H(x,p)$&Lem \ref{Ham1}&Hamiltonian\\
$\Gamma(x,p)$& Lem \ref{gammacoord}&Legendre tranformation\\
$\psi(q)$& Def \ref{psidef}&immersion\\
$\rr^3_+$&Rem \ref{halfspace}& half space\\
$\Psi(x,p)$& Def \ref{cappsidef}&canonical transformation\\
$\rr^1_3$ & Def \ref{cappsidef} &one-dimensional subspace\\
$h(q)$ & Lem \ref{phigammapsi} & metric\\
$\widetilde H(q,\xi)$&Def \ref{Ham2}&Hamiltonian\\
$H_0(q,\xi)$& Def \ref{Ham3}&Hamiltonian\\
$\ww$ & Eq \eqref{invman1}& invariant manifold\\ 
$\rrr(X_3)$ &  Eq \eqref{rotinvman} & rotational invariant manifold\\ %
$\wws$& Eq \eqref{invman2}& stable  manifold\\ %\ref{invman2}
$\wwu$ & Eq \eqref{invman3} &   unstable manifold\\ % \ref{invman3}

\end{tabular}

\clearpage

\bibliography{SL2R}
\bibliographystyle{plain}
\end{document}